\documentclass[11pt,reqno]{amsart}
\usepackage{amsmath,amssymb,latexsym,esint,cite,mathrsfs,dsfont}
\usepackage{verbatim,relsize,color,soul,enumitem}
\usepackage[colorlinks=true,urlcolor=blue, citecolor=red,linkcolor=blue,
linktocpage,pdfpagelabels, bookmarksnumbered,bookmarksopen]{hyperref}
\usepackage[hyperpageref]{backref}

\usepackage{graphicx}
\usepackage{tikz}
\usepackage[font=small,skip=5pt]{caption}
\usepackage[english]{babel}
\usepackage[latin1]{inputenc}
\usepackage{microtype}

\usepackage{geometry}
\numberwithin{equation}{section}
\newtheorem{theorem}{Theorem}[section]
\newtheorem{lemma}[theorem]{Lemma}
\newtheorem{proposition}[theorem]{Proposition}
\newtheorem{corollary}[theorem]{Corollary}

\newtheoremstyle{remarkstyle}
{}{}{}{}{\bfseries}{.}{ }{\thmname{#1}\thmnumber{ #2}\thmnote{ (#3)}}
\theoremstyle{remarkstyle}
\newtheorem{remark}{Remark}[section]
\newtheorem{definition}{Definition}[section]
\newtheorem{observation}{Observation}[section]

\newcommand{\N}{\mathbb N}
\newcommand{\Z}{\mathbb Z}

\newcommand{\R}{\mathbb R}
\newcommand{\C}{\mathbb C}

\newcommand{\Ac}{\mathcal A}
\newcommand{\Bc}{\mathcal B}

\newcommand{\Fc}{\mathcal F}
\newcommand{\Sc}{\mathcal S}
\newcommand{\Vc}{\mathcal V}

\newcommand{\Mcal}{\mathcal M}

\newcommand{\vareps}{\varepsilon}

\newcommand{\scal}[1]{\left\langle #1 \right\rangle}

\DeclareMathOperator*{\loc}{loc}
\DeclareMathOperator*{\rad}{rad}
\DeclareMathOperator*{\opt}{opt}

\DeclareMathOperator*{\rea}{Re}
\DeclareMathOperator*{\ima}{Im}
\DeclareMathOperator*{\sigc}{\sigma_c}

\DeclareMathOperator*{\gamc}{\gamma_c}

\def\({\left(}
\def\){\right)}
\def\<{\left\langle}
\def\>{\right\rangle}

\title[Scattering IBNLS]
{Energy scattering for a class of inhomogeneous biharmonic nonlinear Schr\"odinger equations in low dimensions}

\author[V. D. Dinh]{Van Duong Dinh}
\address[V. D. Dinh]{Ecole Normale Sup\'erieure de Lyon \& CNRS, UMPA (UMR 5669), France}
\email{contact@duongdinh.com}

\author[S. Keraani]{Sahbi Keraani}
\address[S. Keraani]{Laboratoire Paul Painlev\'e UMR 8524, Universit\'e de Lille CNRS, 59655 Villeneuve d'Ascq Cedex, France}
\email{sahbi.keraani@univ-lille.fr}


\begin{document}
	
\begin{abstract}
We consider a class of biharmonic nonlinear Schr\"odinger equations with a focusing inhomogeneous power-type nonlinearity
\[
i\partial_t u -\Delta^2 u+\mu\Delta u +|x|^{-b} |u|^\alpha u=0, \quad \left. u\right|_{t=0}=u_0 \in H^2(\R^d)
\]
with $d\geq 1, \mu\geq 0$, $0<b<\min\{d,4\}$, $\alpha>0$, and $\alpha<\frac{8-2b}{d-4}$ if $d\geq 5$. We first determine a region in which solutions to the equation exist globally in time. We then show that these global-in-time solutions scatter in $H^2(\R^d)$ in three and higher dimensions. In the case of no harmonic perturbation, i.e., $\mu=0$, our result extends the energy scattering proved by Saanouni [Calc. Var. 60 (2021), art. no. 113] and Campos and Guzm\'an [Calc. Var. 61 (2022), art. no. 156] to three and four dimensions. Our energy scattering is new in the presence of a repulsive harmonic perturbation $\mu>0$. The proofs rely on estimates in Lorentz spaces which are properly suited for handling the weight $|x|^{-b}$.   
\end{abstract}
		
\maketitle
	
\section{Introduction}
	\label{S1}
	\setcounter{equation}{0}
	
	Biharmonic (or fourth-order) nonlinear Schr\"odinger equations are equations of the form
	\begin{align} \label{4NLS}
	i \partial_t u - \Delta^2 u+\mu \Delta u = F(x,u) u, \quad (t,x) \in \R\times \R^d
	\end{align} 
	which have been  introduced by Karpman \cite{Karpman} and Karpman--Shagalov \cite{KS} in order to take into consideration the role of small fourth-order dispersion terms in the propagation of laser beams in a bulk medium with Kerr nonlinearity. The function $F(x,u)$ can be understood as a nonlinear potential affected by electron density (see e.g. \cite{Berge}).
	
	\vspace{1mm}
	
	In the case of homogeneous power-type nonlinearity $F(x,u)=\pm|u|^\alpha u$, biharmonic nonlinear Schr\"odinger equations have been studied in many works. In \cite{FIP}, Fibich, Ilan, and Papanicolaou showed the existence of global-in-time $H^2$-solutions and gave numerical simulations supporting the existence of finite time blow-up solutions. Ben-Artzi, Koch, and Saut \cite{BKS} established sharp dispersive estimates for the fourth-order Schr\"odinger operator. Thanks to these estimates, local well-posedness results with $H^2$-initial data have been proved in \cite{Pausader-DPDE} (see also \cite{Dinh-NON, LZ}).  The energy scattering results were investigated in \cite{Pausader-DPDE, Pausader-JFA, Pausader-DCDS, MXZ-JDE-09, MXZ-JDE-11} for the energy-critical nonlinearity, in \cite{PX, Guo, Dinh-NON} for the mass and energy intercritical nonlinearity, and in \cite{PS} for the mass-critical nonlinearity. The existence of finite time blow-up radial $H^2$-solutions was established in \cite{BL, Dinh-NON} (see also a series of works of Fibich et al. \cite{BFM-NON, BFM-SIAM, BF} for numerical analysis of blow-up solutions). Dynamical properties such as mass-concentration and limiting profile of blow-up $H^2$-solutions were studied in \cite{ZZY}. The existence of standing waves as well as their stability/instability were studied in \cite{BCSN, BCGJ, FLNZ}.
	
	\vspace{1mm}
	
	In this paper, we consider the Cauchy problem for biharmonic nonlinear Schr\"odinger equations with a focusing inhomogeneous power-type nonlinearity, namely
	\begin{align} \label{IBNLS}
	\left\{
	\begin{array}{rcl}
	i\partial_t u - \Delta^2 u + \mu \Delta u &=& - |x|^{-b}|u|^\alpha u, \quad (t,x) \in \R \times \R^d, \\
	\left.u\right|_{t=0} &=& u_0 \in H^2(\R^d),
	\end{array}
	\right.
	\end{align}
	where $d\geq 1$, $u:\R \times \R^d \rightarrow \C$, $u_0: \R^d\rightarrow \C$, $\mu\geq 0$, and $\alpha>0$. To our knowledge, the first work addressed this type of equation belongs to Cho, Ozawa, and Wang \cite{COW} where the special case $b=2$ and $\alpha=\frac{4}{d}$ was considered. After proving the existence of local solutions via a regularizing argument and Strichartz estimates, they showed the existence of finite time blow-up solutions with initial data either having finite variance or being radially symmetric. The finite variance and radial symmetry conditions have been removed by the first author in a recent work \cite{Dinh-AML}. Following the influential work \cite{COW}, the inhomogeneous biharmonic nonlinear Schr\"odinger equation has attracted considerable interest from many researchers in the past few years. In \cite{GP-1}, Guzm\'an and Pastor used the contraction mapping argument and Strichartz estimates to prove the local well-posedness (LWP) for \eqref{IBNLS} with $\mu=0$ in dimensions $d\geq 3$. They also obtained conditions for the small data global existence and scattering in the mass and energy intercritical regime. These small data results were recently improved in \cite{GP-2} for dimensions $d=6,7$ and in \cite{CG} for dimension $d=5$. In \cite{LZ}, Liu and Zhang established the local theory for the inhomogeneous biharmonic nonlinear Schr\"odinger equation in Sobolev spaces $H^s(\R^d)$ with $0<s\leq 2$. In particular, their result for $s=2$ improves an earlier LWP of Guzm\'an and Pastor by removing an extra condition on $\alpha$ and extending it to all dimensions $d\geq 1$. The proof in \cite{LZ} relies on a bi-linear Strichartz estimate in Besov spaces. We also mention a recent paper of An, Ryu, and Kim \cite{ARK} for another approach for the local theory using Sobolev-Lorentz spaces. From the above-mentioned works, especially \cite{LZ}, equation \eqref{IBNLS} is known to be locally well-posed. More precisely, for $d\geq 1$, $0<b<\min\left\{\frac{d}{2},4\right\}$, $\alpha>0$, and $\alpha<\frac{8-2b}{d-4}$ if $d\geq 5$, there exist $T^*, T_* \in (0,\infty]$ and a unique solution 
	\[
	u \in C((-T_*,T^*), H^2(\R^d)) \bigcap_{(q,r)\in B} L^q_{\loc}((-T_*,T^*), W^{2,r}(\R^d))
	\]
	to \eqref{IBNLS}, where $(q,r)\in B$ means that $(q,r)$ is a biharmonic admissible pair (see Section \ref{S2} for the definition). The maximal time of existence satisfies the blow-up alternative: if $T^*<\infty$ (resp. $T_*<\infty$), then 
	\[
	\lim_{t\nearrow T^*} \|u(t)\|_{H^2(\R^d)}=\infty \quad \left(\text{resp. } \lim_{t\searrow -T_*} \|u(t)\|_{H^2(\R^d)}=\infty\right).
	\]
	In addition, there are conservation laws of mass and energy
	\[
	M(u(t))= \int |u(t,x)|^2 dx = M(u_0) \tag{Mass}
	\]
	and
	\[
	E_\mu(u(t))= \frac{1}{2}\int |\Delta u(t,x)|^2 dx +\frac{\mu}{2}\int  |\nabla u(t,x)|^2 dx - \frac{1}{\alpha+2} \int |x|^{-b} |u(t,x)|^{\alpha+2} dx = E_\mu(u_0). \tag{Energy}
	\]
	In the case of no harmonic term (i.e., $\mu =0$), equation \eqref{IBNLS} has a scaling invariance
	\[
	u_\lambda(t,x) = \lambda^{\frac{4-b}{\alpha}} u(\lambda^4 t, \lambda x), \quad \lambda>0.
	\] 
	A direct computation shows that this scaling leaves the $\dot{H}^{\gamc}$-norm of initial data invariant, i.e., $\|u_\lambda(0)\|_{\dot{H}^{\gamc}(\R^d)} = \|u_0\|_{\dot{H}^{\gamc}(\R^d)}$, where
	\begin{align}\label{gamc}
	\gamc:= \frac{d}{2}-\frac{4-b}{\alpha}.
	\end{align}
	When $\gamc=0$ or $\alpha=\frac{8-2b}{d}$, equation \eqref{IBNLS} is called mass-critical. When $\gamc=2$ or $\alpha=\frac{8-2b}{d-4}$ and $d\geq 5$, equation \eqref{IBNLS} is called energy-critical. 
	
	\vspace{1mm}
	
	The main purpose of this paper is to show the energy scattering for \eqref{IBNLS} with a mass-supercritical and energy-subcritical nonlinearity, i.e., $0<\gamc<2$. For the reader convenience, we recall the notion of energy scattering as follows: a solution $u(t)$ to \eqref{IBNLS} scatters in $H^2(\R^d)$ forward in time (resp. backward in time) if there exists $u_+ \in H^2(\R^d)$ (resp. $u_- \in H^2(\R^d)$) such that 
	\begin{align}\label{def-scat}
	\lim_{t\to + \infty} \|u(t) - U_\mu(t) u_+\|_{H^2(\R^d)}=0 \quad \left(\text{resp. } \lim_{t\to - \infty} \|u(t) - U_\mu(t) u_-\|_{H^2(\R^d)}=0\right),
	\end{align}
	where 
	\[
	U_\mu(t):= e^{-it(\Delta^2-\mu\Delta)}
	\] 
	is the biharmonic Schr\"odinger propagator. The energy scattering in the mass and energy intercritical regime was first proved by Saanouni \cite{Saanouni} for radial data in dimensions $d\geq5$. The proof is based on an idea of Dodson and Murphy \cite{DM} using a scattering criterion and localized virial estimates of radial solutions. The Dodson--Murphy's approach was the first attempt to give an alternative proof for the radial energy scattering of the cubic nonlinear Schr\"odinger equation without using the concentration-compactness-rigidity method of Kenig and Merle (see \cite{HR}). In \cite{CG}, Campos and Guzm\'an improved Saanouni's result by removing the radial assumption, but still in dimensions $d\geq 5$. Their proof relies on Murphy's idea \cite{Murphy} which is a combination of scattering criterion and Virial-Morawetz type estimates. Here instead of relying on the radial assumption, they exploited the spatial decay at infinity of nonlinearity. As mentioned in \cite{CG}, the restriction $d\geq 5$ comes from the decay of linear biharmonic operator $|t|^{-\frac{d}{4}}$ which is not integrable on $[1,\infty)$ when $d\leq 4$. Therefore, the energy scattering for \eqref{IBNLS} in dimensions $d\leq 4$ was raised as an open problem (see \cite[Remark 1.4]{CG}). We stress that the energy scattering results in \cite{Saanouni, CG} are only for \eqref{IBNLS} with $\mu=0$, i.e., there is no Laplacian perturbation term. Our purpose is twofold. First, we establish the energy scattering for $\mu=0$ in the whole range of intercritical regime in low dimensions $d=3,4$. Second, we extend the energy scattering to all $\mu>0$ and dimensions $d\geq 3$.  
	
	\vspace{1mm}
	
	To put our work in context, we introduce the minimization problem: for $\omega>0$,
	\begin{align} \label{m-mu-omega}
	m_{\mu,\omega}:= \inf \left\{S_{\mu,\omega}(f) : f \in H^2(\R^d)\backslash \{0\}, G_\mu(f)=0\right\},
	\end{align}
	where
	\[
	S_{\mu,\omega}(f):= E_\mu(f)+\frac{\omega}{2} M(f) = \frac{1}{2}\|\Delta f\|^2_{L^2}+\frac{\mu}{2}\|\nabla f\|^2_{L^2} +\frac{\omega}{2}\|f\|^2_{L^2} -\frac{1}{\alpha+2}\int |x|^{-b}|f(x)|^{\alpha+2}dx
	\]
	is the action functional and 
	\[
	G_\mu(f):= 2\|\Delta f\|^2_{L^2}+\mu \|\nabla f\|^2_{L^2} -\frac{d\alpha+2b}{2(\alpha+2)}\int |x|^{-b} |f(x)|^{\alpha+2}dx
	\]
	is the Pohozaev functional. We will see in Proposition \ref{prop-mini-prob} that $m_{\mu,\omega}>0$ and there exists at least a minimizer for $m_{\mu,\omega}$ which solves the elliptic equation
	\begin{align} \label{elli-equa}
	\Delta^2 Q-\mu\Delta Q +\omega Q -|x|^{-b} |Q|^\alpha Q=0.
	\end{align}
	We define the sets
	\begin{align}\label{Ac-omega}
	\begin{aligned}
	\Ac^+_{\mu,\omega}:&= \left\{f\in H^2(\R^d): S_{\mu,\omega}(f)<m_{\mu,\omega}, G_\mu(f)\geq 0\right\}, \\
	\Ac^-_{\mu,\omega}:&= \left\{f\in H^2(\R^d): S_{\mu,\omega}(f)<m_{\mu,\omega}, G_\mu(f)< 0\right\}.
	\end{aligned}
	\end{align}
	Our first result concerns the global existence and space-time estimates with initial data in $\Ac^+_{\mu,\omega}$. 
	\begin{theorem}[\bf Global existence and space-time estimates]   \label{theo-gwp} ~ \\
		Let $d\geq 1$, $\mu\geq 0$, $0<b<\min\left\{\frac{d}{2},4\right\}$, $\alpha>\frac{8-2b}{d}$, $\alpha<\frac{8-2b}{d-4}$ if $d\geq 5$, and $\omega>0$. Let $u_0 \in \Ac^+_{\mu,\omega}$. Then the corresponding solution to \eqref{IBNLS} exists globally in time and satisfies for any time interval $I\subset \R$,
		\begin{align} \label{space-time-est-non-rad}
		\int_I \int_{\R^d} |x|^{-b} |u(t,x)|^{\alpha+2} dx dt \leq C |I|^{\frac{1}{1+\min\{2,b\}}}.
		\end{align}
		In addition, if $u_0$ is radial, then we have for any time interval $I\subset \R$,
		\begin{align} \label{space-time-est-rad}
		\int_I \int_{\R^d} |x|^{-b} |u(t,x)|^{\alpha+2} dx dt \leq C |I|^{\frac{1}{3}}.
		\end{align}
	\end{theorem}

	\begin{remark}
		Solutions to \eqref{IBNLS} with initial data in $\Ac^-_{\mu,\omega}$ have different behaviors compared to those in $\Ac^+_{\mu,\omega}$. It will be investigated in a forthcoming work. 
	\end{remark}
	
	Our second result deals with the energy scattering for global-in-time solutions obtained in Theorem \ref{theo-gwp}.
	\begin{theorem}[\bf Energy scattering] \label{theo-scat} ~ \\ 
		Let $u \in C(\R, H^2(\R^d))$ be the global-in-time solution obtained in Theorem \ref{theo-gwp}. Assume in addition that one of the following conditions is satisfied:
		\begin{itemize} [leftmargin=6mm]
			\item[(1)] $d\geq 5$; or $d=4$ and $1<b<2$;
			\item[(2)] $d=4$ and $0<b\leq 1$; or $d=3$ and $0<b<\frac{3}{2}$; and $u_0$ is radial.
		\end{itemize}
		Then the solution scatters in $H^2(\R^d)$ in both directions. 
	\end{theorem}
	
	\begin{remark} ~ 
		\begin{itemize}[leftmargin=6mm]
			\item[(1)] When $\mu=0$, the threshold for energy scattering in \cite{Saanouni, CG} was determined in terms of mass and energy of the ground state solution to 
			\[
			\Delta^2Q - Q-|x|^{-b}|Q|^{\alpha} Q=0.
			\]
			In Appendix \ref{appendix}, we prove that the threshold given by $\Ac^+_{0,\omega}$ and the one in \cite{Saanouni,CG} are identical. Therefore, our result is indeed an extension of those in \cite{Saanouni, CG} to low dimensions.
			\item[(2)] When $\mu>0$, there are no energy scattering result available according to our knowledge. Hence our result is new in this case.
			\item[(3)] The condition $\mu\geq 0$ is essential for the energy scattering since for $\mu<0$, Strichartz estimates for $U_\mu(t)$ only hold locally in time.
			\item[(4)] When $d=4$ and $0<b\leq 1$; or $d=3$ and $0<b<\frac{3}{2}$, the growth of non-radial space-time estimate \eqref{space-time-est-non-rad} is too strong in comparison to the decay $|t|^{-\frac{d}{4}}$ of the linear operator and we are not able to show the smallness of the distance past (see below). The problem is resolved if we restrict to radial solutions. The non-radial energy scattering for $d=4$ and $0<b\leq 1$ as well as for $d=1, 2, 3$ still remains an open question.
		\end{itemize} 
	\end{remark}
	
	Let us briefly describe the strategy of our proofs. To prove Theorem \ref{theo-gwp}, we first study the variational problem $m_{\mu,\omega}$. By making use of the spatial decay of the weight $|x|^{-b}$, we show a compact embedding $H^2(\R^d) \hookrightarrow L^{\alpha+2}(|x|^{-b} dx)$ for $d\geq 1$, $0<b<\min\{d,4\}$, $\alpha>0$, and $\alpha<\frac{8-2b}{d-4}$ if $d\geq 5$. This together with some variational arguments, we are able to show the existence of minimizers for $m_{\mu,\omega}$ (see Proposition \ref{prop-mini-prob}). From this, we deduce that the set $\Ac^+_{\mu,\omega}$ is invariant under the flow of \eqref{IBNLS}. Moreover, solutions to \eqref{IBNLS} with initial data in $\Ac^+_{\mu,\omega}$ enjoy a uniform bound of $H^2$-norm which together with the blow-up alternative yield the global existence. The proof of space-time estimates \eqref{space-time-est-non-rad} and \eqref{space-time-est-rad} are more involved. The first step is to prove a coercivity property of $\Ac^+_{\mu,\omega}$ which is stated in Proposition \ref{prop-coer}. The second step is to show that for solutions $u(t)$ with initial data in $\Ac^+_{\mu,\omega}$, the multiplication $\chi_R u(t)$ still remains in $\Ac^+_{\mu,\omega}$ for large $R$, where $\chi_R$ is a smooth cutoff which equals to 1 on the ball centered at zero and of radius $R/2$. This allows us to exhibit the coercivity property for $\chi_R u(t)$ for $R$ large, namely
	\[
	G_\mu(\chi_R u(t)) \geq C \int_{\R^d} |x|^{-b} |\chi_R(x)u(t,x)|^{\alpha+2}dx, \quad \forall t\in \R.
	\] 
	The last step is to perform a localized virial estimate and insert the cutoff into Pohozaev functional. Here we make intensive use of the decay property of $|x|^{-b}$ outside a large ball.  
	
	\vspace{1mm}
	
	The proof of Theorem \ref{theo-scat} is inspired by an idea of Arora, Dodson, and Murphy \cite{ADM} which is outlined as follows. After some reductions, the problem is reduced to prove that for $\vareps>0$ small, there exist $T=T(\vareps)>0$ sufficiently large and $t_1 \in (0,T)$ so that
	\begin{align} \label{est-T-intro}
	\left\|\int_0^{t_1} U_\mu(t-\tau) (|x|^{-b} |u|^\alpha u)(\tau) d\tau\right\|_{L^k([t_1, +\infty), L^{r,2}(\R^d))} \leq C \vareps^\gamma
	\end{align}
	for a suitable pair $(k,r)$ and some $\gamma>0$, where $L^{r,2}(\R^d)$ stands for the Lorentz space (see Section \ref{S2} for the definition). To show this, we first use the space-time estimates to find $t_0 \in \left[\frac{T}{4},\frac{T}{2}\right]$ so that
	\begin{align} \label{est-t0-intro}
	\int_{t_0}^{t_0+\vareps T^{1-\varrho}} \int_{\R^d} |x|^{-b} |u(t,x)|^{\alpha+2} dx dt \leq C\vareps,
	\end{align}
	where $\varrho = \frac{1}{1+\min\{2,b\}}$ for general solutions and $\varrho = \frac{1}{3}$ for radial solutions. By choosing $t_1=t_0+\vareps T^{1-\varrho}$, we split the integral in the left hand side of \eqref{est-T-intro} into two parts: the distant past
	\[
	F_1(t) = \int_0^{t_0} U_\mu(t-\tau) (|x|^{-b} |u|^\alpha u)(\tau) d\tau
	\]
	and the recent past
	\[
	F_2(t) = \int_{t_0}^{t_1} U_\mu(t-\tau) (|x|^{-b} |u|^\alpha u)(\tau) d\tau.
	\]
	To estimate the distant past, we use dispersive estimates in Lorentz spaces. For $d\geq 5$, the decay rate $(t-\tau)^{-\frac{d}{4}}$ is integrable for $\tau \in [t_1,+\infty)$ which allows us to pick $T=T(\vareps)$ large to ensure the smallness of the distant past. For $d=3,4$, the above decay is no longer integrable on $[t_1,+\infty)$. Thus we have to find suitable nonlinear estimates to control this decay in higher Lebesgue norm on $[t_1, +\infty)$. This requires the growth of space-time estimate is not too big. In particular, we need $\varrho<\frac{1}{2}$ in four dimensions which leads to the restriction $1<b<2$ for the non-radial scattering. In three dimensions, we need $\varrho$ even smaller which prevents us to show the same non-radial scattering for $0<b<\frac{3}{2}$. However, for radial solutions, the growth $\varrho=\frac{1}{3}$ is enough to make the distant past small. To estimate the recent past, we prove a nonlinear estimate which allows us to insert the smallness \eqref{est-t0-intro} to control the recent past.  
	
	To make the above argument possible, we provide some new features which are elaborated as follows. First, instead of working on Lebesgue spaces where nonlinear estimates were proved by considering separately two regions inside and outside a ball (see \cite{GP-1,GP-2}), we use Lorentz spaces which are well suitable for estimating the weighted term $|x|^{-b}$. By using these spaces, we are able to prove nonlinear estimates in a unified way (see Lemmas \ref{lem-non-est-1} and \ref{lem-non-est-2}) which contains much simpler proofs than those in \cite{GP-1, GP-2}. In the context of inhomogeneous nonlinear Schr\"odinger equation, Lorentz spaces were first used in \cite{AT} to show the local well-posedness in Sobolev spaces. These Lorentz spaces were also exploited recently in \cite{ARK} to prove the local theory for inhomogeneous fourth-order Schr\"odinger equation. Second, we prove Strichartz estimates for the biharmonic Schr\"odinger operator in Lorentz spaces. By interpolating between the $L^2$-isometry and the $L^1-L^\infty$ estimate of the linear operator, we obtain dispersive estimates in Lorentz spaces. The homogeneous and inhomogeneous Strichartz estimates follow from an abstract argument of Keel and Tao \cite{KT}. However, to treat lower dimensions, we also need Strichartz estimates with a gain of derivatives in Lorentz spaces which are proved in Lemma \ref{lem-str-est-gain}. 
	
	Finally, we remark that the argument developed in this paper can be applied to show the energy scattering for the defocusing inhomogeneous biharmonic nonlinear Schr\"odinger equations IBNLS (i.e., the plus sign in front of the nonlinearity in \eqref{IBNLS}). In particular, we have the following result.
	
	\begin{proposition}[\bf Energy scattering for defocusing IBNLS] \label{prop-scat-defo} ~\\
		Let $d\geq 1$, $\mu\geq 0$, $0<b<\min\left\{\frac{d}{2},4\right\}$, $\alpha>\frac{8-2b}{d}$, $\alpha<\frac{8-2b}{d-4}$ if $d\geq 5$, and $u_0\in H^2(\R^d)$. Then the corresponding solution to the defocusing IBNLS exists globally in time. Assume in addition that one of the following conditions is satisfied:
		\begin{itemize} [leftmargin=6mm]
			\item[(1)] $d\geq 5$; or $d=4$ and $1<b<2$;
			\item[(2)] $d=4$ and $0<b\leq 1$; or $d=3$ and $0<b<\frac{3}{2}$; and $u_0$ is radial.
		\end{itemize}
		Then the solution scatters in $H^2(\R^d)$ in both directions. 
	\end{proposition}
	
	\begin{remark}
	Proposition \ref{prop-scat-defo} improves a previous work of Saanouni \cite{Saanouni-PA} where the energy scattering for the defocusing IBNLS with $\mu=0$ was proved for radial solutions in dimensions $d\geq 5$.
	\end{remark}

	The paper is organized as follows. In Section \ref{S2}, we first recall the definition as well as basic properties of Lorentz spaces. We then prove dispersive and Strichartz estimates in Lorentz spaces. In Section \ref{S3}, we show some nonlinear estimates which are needed in our analysis. In Section \ref{S4}, we study the variational problem $m_{\mu,\omega}$ and establish the coercivity property of the set $\Ac^+_{\mu,\omega}$. Section \ref{S5} is devoted to the proofs of our main theorems.  Finally, an equivalent characterization of thresholds will be given in Appendix \ref{appendix}.

\section{Preliminary}
\label{S2}
\setcounter{equation}{0}

\subsection{Lorentz spaces}
Let $f$ be a measurable function on $\R^d$. The distribution function of $f$ is defined by
\[
d_f(\lambda):= |\{x\in \R^d : |f(x)|>\lambda\}|, \quad \lambda>0,
\]
where $|A|$ is the Lebesgue measure of a set $A$ in $\R^d$. The decreasing rearrangement of $f$ is defined by
\[
f^*(s):= \inf \left\{ \lambda>0 : d_f(\lambda)\leq s\right\}, \quad s>0.
\]

\begin{definition}[\bf Lorentz spaces] ~\\
	Let $0<r<\infty$ and $0<\rho\leq \infty$. The Lorentz space $L^{r,\rho}(\R^d)$ is defined by
	\[
	L^{r,\rho}(\R^d):= \left\{ f \text{ measurable on } \R^d : \|f\|_{L^{r,\rho}}<\infty\right\},
	\]
	where
	\[
	\|f\|_{L^{r,\rho}}:= \left\{
	\begin{array}{cl}
	\( \frac{\rho}{r} \mathlarger{\int}_0^\infty \(s^{1/r} f^*(s)\)^\rho \frac{1}{s}ds\)^{1/\rho} &\text{ if } \rho <\infty, \\
	\sup_{s>0} s^{1/r} f^*(s) &\text{ if } \rho=\infty.
	\end{array}
	\right.
	\]
\end{definition}

We collect the following basic properties of $L^{r,\rho}(\R^d)$ in the following lemma (see e.g., \cite{Grafakos}).
\begin{lemma}[\bf Properties of Lorentz spaces] ~ 
	\begin{itemize}[leftmargin=6mm]
		\item For $1<r<\infty$, $L^{r,r}(\R^d) \equiv L^r(\R^d)$ and by convention, $L^{\infty,\infty}(\R^d)= L^\infty(\R^d)$.
		\item For $0<r<\infty$ and $0<\rho \leq \infty$, $L^{r,\rho}(\R^d)$ is a quasi-Banach space.
		\item For $0<r<\infty$ and $0<\rho_1<\rho_2\leq \infty$, $L^{r,\rho_1}(\R^d)\subset L^{r,\rho_2}(\R^d)$ and there exists $C>0$ such that $\|f\|_{L^{r,\rho_2}}\leq C\|f\|_{L^{r,\rho_1}}$ for all $f\in L^{r,\rho_1}(\R^d)$.
		\item For $0<r<\infty$, $0<\rho \leq \infty$, and $\theta>0$, $\||f|^\theta\|_{L^{r,\rho}} = \|f\|^\theta_{L^{\theta r, \theta \rho}}$.
		\item For $1<r<\infty$ and $1\leq \rho \leq \infty$, $L^{r,\rho}(\R^d)$ can be normed to become a Banach space. More precisely, $(L^{r,\rho}(\R^d),\|\cdot\|^*_{L^{r,\rho}})$ is a Banach space with
		\[
		\|f\|^*_{L^{r,\rho}}:= \left\{
		\begin{array}{cl}
		\(\frac{\rho}{r}\mathlarger{\int}_0^\infty \(s^{1/r} f^{**}(s)\)^\rho \frac{1}{s}ds\)^{1/\rho} &\text{ if } \rho <\infty, \\
		\sup_{s>0} s^{1/r}f^{**}(s) &\text{ if } \rho=\infty,
		\end{array}
		\right.
		\]
		where
		\[
		f^{**}(t):= \frac{1}{t} \int_0^t f^*(s)ds, \quad t>0.
		\]
		Moreover, we have the norm equivalence
		\[
		\|f\|_{L^{r,\rho}}\leq \|f\|^*_{L^{r,\rho}}\leq \frac{r}{r-1} \|f\|_{L^{r,\rho}}.
		\]
		\item For $b>0$, $|x|^{-b} \in L^{\frac{d}{b},\infty}(\R^d)$ and $\||x|^{-b}\|_{L^{\frac{d}{b},\infty}} = |B(0,1)|^{\frac{b}{d}}$, where $B(0,1)$ is the unit ball of $\R^d$.
	\end{itemize}
\end{lemma}

\begin{lemma}[\bf H\"older's inequality \cite{ONeil}] ~
	\begin{itemize}[leftmargin=6mm]
		\item Let $1<r, r_1, r_2<\infty$ and $1\leq \rho, \rho_1, \rho_2 \leq \infty$ be such that 
		\[
		\frac{1}{r}=\frac{1}{r_1}+\frac{1}{r_2}, \quad \frac{1}{\rho} \leq \frac{1}{\rho_1}+\frac{1}{\rho_2}.
		\]
		Then there exists $C>0$ such that 
		\[
		\|fg\|_{L^{r,\rho}} \leq C\|f\|_{L^{r_1, \rho_1}} \|g\|_{L^{r_2,\rho_2}}
		\]
		for any $f \in L^{r_1, \rho_1}(\R^d)$ and $g\in L^{r_2, \rho_2}(\R^d)$.
		\item Let $1<r_1, r_2<\infty$ and $1\leq \rho_1, \rho_2 \leq \infty$ be such that
		\[
		1=\frac{1}{r_1}+\frac{1}{r_2}, \quad 1\leq \frac{1}{\rho_1}+\frac{1}{\rho_2}.
		\]
		Then there exists $C>0$ such that
		\[
		\|fg\|_{L^1} \leq C\|f\|_{L^{r_1, \rho_1}} \|g\|_{L^{r_2,\rho_2}}
		\]
		for any $f\in L^{r_1, \rho_1}(\R^d)$ and $g\in L^{r_2, \rho_2}(\R^d)$.
	\end{itemize}
\end{lemma}

\begin{lemma}[\bf Convolution inequality \cite{ONeil}] ~
	\begin{itemize}[leftmargin=6mm]
		\item Let $1<r,r_1, r_2<\infty$ and $1\leq \rho, \rho_1, \rho_2 \leq \infty$ be such that
		\[
		1+\frac{1}{r} =\frac{1}{r_1}+\frac{1}{r_2}, \quad \frac{1}{\rho}\leq \frac{1}{\rho_1}+\frac{1}{\rho_2}.
		\]
		Then there exists $C>0$ such that
		\[
		\|f\ast g\|_{L^{r,\rho}} \leq C\|f\|_{L^{r_1,\rho_1}} \|g\|_{L^{r_2,\rho_2}}
		\]
		for any $f\in L^{r_1, \rho_1}(\R^d)$ and $g\in L^{r_2, \rho_2}(\R^d)$.
		\item Let $1<r_1, r_2<\infty$ and $1\leq \rho_1, \rho_2 \leq \infty$ be such that 
		\[
		1=\frac{1}{r_1}+\frac{1}{r_2}, \quad 1\leq \frac{1}{\rho_1}+\frac{1}{\rho_2}.
		\]
		Then there exists $C>0$ such that
		\[
		\|f\ast g\|_{L^\infty} \leq C\|f\|_{L^{r_1, \rho_1}} \|g\|_{L^{r_2,\rho_2}}
		\]
		for any $f\in L^{r_1, \rho_1}(\R^d)$ and  $g\in L^{r_2, \rho_2}(\R^d)$.
	\end{itemize}
\end{lemma}

Let $s\geq 0$, $1<r<\infty$, and $1\leq \rho \leq \infty$. We define the Sobolev-Lorentz spaces
\begin{align*}
W^sL^{r,\rho}(\R^d) &=\left\{ f \in \Sc'(\R^d) : (I-\Delta)^{s/2} f\in L^{r,\rho}(\R^d)\right\}, \\
\dot{W}^s L^{r,\rho}(\R^d) &= \left\{f \in \Sc'(\R^d) : (-\Delta)^{s/2} f\in L^{r,\rho}(\R^d) \right\},
\end{align*}
where $\Sc'(\R^d)$ is the space of tempered distributions on $\R^d$ and
\[
(I-\Delta)^{s/2} f = \Fc^{-1}\( (1+|\xi|^2)^{s/2} \Fc(f)\), \quad (-\Delta)^{s/2} f = \Fc^{-1}\( |\xi|^s \Fc(f)\)
\]
with $\Fc$ and $\Fc^{-1}$ the Fourier and its inverse Fourier transforms respectively. The spaces $W^sL^{r,\rho}(\R^d)$ and $\dot{W}^sL^{r,\rho}(\R^d)$ are endowed respectively with the norms
\[
\|f\|_{W^sL^{r,\rho}} = \|f\|_{L^{r,\rho}} + \|(-\Delta)^{s/2} f\|_{L^{r,\rho}}, \quad \|f\|_{\dot{W}^sL^{r,\rho}} = \|(-\Delta)^{s/2} f\|_{L^{r,\rho}}.
\]

\begin{corollary}[\bf Sobolev embedding] \label{coro-sobo-embe} ~ \\
	Let $1<r<\infty$, $1\leq \rho\leq \infty$, and $0<s<\frac{d}{r}$. Then there exists $C>0$ such that
	\[
	\|f\|_{L^{\frac{dr}{d-s r}, \rho}} \leq C\|(-\Delta)^{s/2} f\|_{L^{r,\rho}}
	\]
	for any $f\in \dot{W}^sL^{r,\rho}(\R^d)$. In particular, we have 
	\[
	W^sL^{r,\rho}(\R^d) \subset L^{n,\rho}(\R^d), \quad r\leq n \leq \frac{dr}{d-sr}.
	\]
\end{corollary}
\begin{proof}
	We have
	\[
	(-\Delta)^{-s/2} f(x) =C(d,s) \int_{\R^d} \frac{f(y)}{|x-y|^{d-s}} dy = C(d,s) (|\cdot|^{-(d-s)}\ast f)(x).
	\]
	By the convolution inequality in Lorentz spaces and the fact that $|x|^{-(d-s)} \in L^{\frac{d}{d-s},\infty}(\R^d)$, we get
	\[
	\|(-\Delta)^{-s/2}f\|_{L^{\frac{dr}{d-s r}, \rho}} \leq C \|f\|_{L^{r,\rho}}
	\]
	which proves the desired estimate.
\end{proof}
\begin{corollary}[\bf Hardy's inequality] ~ \\
	Let $1<r<\infty$, $1\leq \rho\leq \infty$, and $0<s<\frac{d}{r}$. Then there exists $C>0$ such that
	\[
	\||x|^{-s} f\|_{L^{r,\rho}} \leq C\|(-\Delta)^{s/2} f\|_{L^{r,\rho}}
	\]
	for any $f\in \dot{W}^sL^{r,\rho}(\R^d)$.
\end{corollary}
\begin{proof}
	By H\"older's inequality and Sobolev embedding in Lorentz spaces, we have
	\[
	\||x|^{-s} f\|_{L^{r,\rho}} \leq \||x|^{-s}\|_{L^{\frac{d}{s},\infty}} \|f\|_{L^{\frac{dr}{d-sr}, \rho}} \leq C \|(-\Delta)^{s/2} f\|_{L^{r,\rho}}.
	\]
	This gives the result.
\end{proof}

\begin{lemma}[\bf Product rule \cite{CN}] ~\\
	Let $s\geq 0$, $1<r, r_1,r_2, r_3, r_4<\infty$, and $1\leq \rho, \rho_1, \rho_2, \rho_3, \rho_4\leq \infty$ be such that
	\[
	\frac{1}{r}= \frac{1}{r_1}+\frac{1}{r_2}=\frac{1}{r_3}+\frac{1}{r_4}, \quad \frac{1}{\rho}=\frac{1}{\rho_1}+\frac{1}{\rho_2}=\frac{1}{\rho_3}+\frac{1}{\rho_4}.
	\]
	Then there exists $C>0$ such that
	\[
	\|(-\Delta)^{s/2}(fg)\|_{L^{r,\rho}} \leq C\( \|(-\Delta)^{s/2} f\|_{L^{r_1, \rho_1}} \|g\|_{L^{r_2, \rho_2}} + \|f\|_{L^{r_3,\rho_3}} \|(-\Delta)^{s/2} g\|_{L^{r_4,\rho_4}} \)
	\]
	for any $f\in \dot{W}^sL^{r_1, \rho_1}(\R^d) \cap L^{r_3,\rho_3}(\R^d)$ and $g\in \dot{W}^s L^{r_4,\rho_4}(\R^d) \cap L^{r_2,\rho_2}(\R^d)$.
\end{lemma}

\begin{lemma}[\bf Chain rule \cite{AT}] ~\\
	Let $0\leq s\leq 1$, $F \in C^1(\C,\C)$, $1<r,r_1, r_2<\infty$, and $1\leq \rho, \rho_1, \rho_2 <\infty$ be such that
	\[
	\frac{1}{r} = \frac{1}{r_1}+\frac{1}{r_2}, \quad \frac{1}{\rho}=\frac{1}{\rho_1}+\frac{1}{\rho_2}.
	\]
	Then there exists $C>0$ such that
	\[
	\|(-\Delta)^{s/2} F(f)\|_{L^{r,\rho}} \leq C\|F'(f)\|_{L^{r_1, \rho_1}} \|(-\Delta)^{s/2} f\|_{L^{r_2, \rho_2}}.
	\]
\end{lemma}


\subsection{Dispersive and Strichartz estimates in Lorentz spaces}
\begin{lemma}[\bf Dispersive estimates] ~\\
	Let $d\geq 1$, $\mu\geq 0$, $2<r<\infty$, and $1\leq \rho \leq \infty$. Then there exists $C>0$ such that
	\begin{align} \label{dis-est-lore}
	\|U_\mu(t) f\|_{L^{r,\rho}_x} \leq C|t|^{-\frac{d}{4}\(1-\frac{2}{r}\)} \|f\|_{L^{r',\rho}_x}, \quad t\ne 0
	\end{align}
	for any $f\in L^{r',\rho}(\R^d)$.
\end{lemma}
\begin{proof}
	We first have the unitary property
	\[
	\|U_\mu(t)f\|_{L^2_x} = \|f\|_{L^2_x}, \quad \forall t \in \R.
	\]
	We also recall the following standard dispersive estimates (see \cite{BKS})
	\begin{align} \label{dis-est}
	\|U_\mu(t) f\|_{L^\infty_x} \leq C|t|^{-\frac{d}{4}} \|f\|_{L^1_x}, \quad \forall t\ne 0.
	\end{align}
	Note that \eqref{dis-est} holds for all time provided that $\mu\geq 0$. When $\mu<0$, such an estimate only holds for short time.  
	The desired estimate \eqref{dis-est-lore} follows by applying the interpolation theorem (see e.g., \cite[Theorem 1.4.19]{Grafakos}) with $T=U_\mu(t)$, $q_0=2, q_1=\infty, p_0=2, p_1=1$, $\theta =1-\frac{2}{r}$ so that
	\[
	\frac{1}{r}=\frac{1-\theta}{q_0}+\frac{\theta}{q_1}, \quad \frac{1}{r'}=\frac{1-\theta}{p_0} +\frac{\theta}{p_1}
	\]
	and $M_0=1, M_1=C|t|^{-\frac{d}{4}}$.
\end{proof}

\begin{definition}[Admissibility] ~
	\begin{itemize}[leftmargin=6mm]
	\item A pair $(q,r)$ is said to be biharmonic admissible, for short $(q,r)\in B$, if
	\[
	\frac{4}{q}+\frac{d}{r}=\frac{d}{2}, \quad \left\{
	\begin{array}{ll}
	r\in \left[2,\frac{2d}{d-4}\right] &\text{if } d\geq 5,\\
	r\in [2,\infty) &\text{if } d=4, \\
	r\in [2, \infty] &\text{if } d=1,2,3.
	\end{array}
	\right.
	\]
	\item A pair $(p,n)$ is said to be Sch\"odinger admissible, for short $(p,n)\in S$, if 
	\[
	\frac{2}{p}+\frac{d}{n}=\frac{d}{2}, \quad \left\{
	\begin{array}{ll}
	n\in \left[2,\frac{2d}{d-2}\right] &\text{if } d\geq 3,\\
	n\in [2,\infty) &\text{if } d=2, \\
	n\in [2, \infty] &\text{if } d=1.
	\end{array}
	\right.
	\]
	\end{itemize}
\end{definition}

Thanks to dispersive estimates \eqref{dis-est-lore}, the $TT^\ast$-argument yields the following Strichartz estimates (see \cite{KT}).
\begin{proposition}[\bf Strichartz estimates] ~ \\
	Let $d\geq 1$ and $\mu\geq 0$.  
	\begin{itemize}[leftmargin=6mm]
		\item Let $(q,r)\in B$ with $r<\infty$. Then there exists $C>0$ such that
		\begin{align} \label{str-est-homo}
		\|U_\mu(t) f\|_{L^q_t(\R, L^{r,2}_x)} \leq C\|f\|_{L^2_x}
		\end{align}
		for any $f\in L^2(\R^d)$.
		\item Let $(q_1, r_1), (q_2,r_2)\in B$ with $r_1, r_2<\infty$, $t_0\in \R$, and $I\subset \R$ be an interval containing $t_0$. Then there exists $C>0$ such that
		\begin{align}\label{str-est-inhomo}
		\left\|\int_{t_0}^t U_\mu(t-\tau) F(\tau) d\tau\right\|_{L^{q_1}_t(I, L^{r_1,2}_x)} \leq C\|F\|_{L^{q_2'}_t(I, L^{r_2',2}_x)}
		\end{align}
		for any $F\in L^{q_2'}(I, L^{r_2',2}(\R^d))$.
	\end{itemize}
\end{proposition}

\begin{lemma}[\bf Strichartz estimates with a gain of derivatives] \label{lem-str-est-gain} ~\\
	Let $d\geq 1$, $\mu\geq 0$, $(p_1,n_1), (p_2,n_2) \in S$, $t_0 \in \R$, and $I\subset \R$ be an interval containing $t_0$. Then there exists $C>0$ such that
	\begin{align}\label{str-est-schr}
	\left\|\Delta \int_{t_0}^t U_\mu(t-\tau) F(\tau) d\tau \right\|_{L^{p_1}_t(I, L^{n_1,2}_x)} \leq C \||\nabla|^{2-\frac{2}{p_1}-\frac{2}{p_2}} F\|_{L^{p'_2}_t(I, L^{n'_2,2}_x)}
	\end{align}
	for any $F \in L^{p'_2}(I,\dot{W}^{2-\frac{2}{p_1}-\frac{2}{p_2}}L^{n'_2,2}(\R^d))$, where $|\nabla|:= (-\Delta)^{1/2}$. In particular, for $(q,r)\in B$ with $r<\infty$ and $(p,n)\in S$ with $n<\infty$, we have 
	\begin{align} \label{str-est-gain}
	\left\|\Delta \int_{t_0}^t U_\mu(t-\tau) F(\tau) d\tau \right\|_{L^q_t(I, L^{r,2}_x)} \leq C \||\nabla|^{2-\frac{2}{p}} F\|_{L^{p'}_t(I, L^{n',2}_x)}
	\end{align}
	for any $F \in L^{p'}(I,\dot{W}^{2-\frac{2}{p}}L^{n',2}(\R^d))$.
\end{lemma}
\begin{proof}
	We first show how to get \eqref{str-est-gain} from \eqref{str-est-schr}. We take $\overline{r}$ such that $\frac{d}{\overline{r}}=\frac{d}{r}+\frac{2}{q}$. By Sobolev embedding in Lorentz spaces (Corollary \ref{coro-sobo-embe}), we have
	\[
	\|f\|_{L^{r,2}_x} \leq C \||\nabla|^{\frac{2}{q}} f\|_{L^{\overline{r},2}_x}.
	\]
	Since $(q,\overline{r})\in S$, we apply \eqref{str-est-schr} with $|\nabla|^{\frac{2}{q}} F$ in place of $F$ and get
	\begin{align*}
	\left\|\Delta \int_{t_0}^t U_\mu(t-\tau) F(\tau) d\tau \right\|_{L^q_t(I,L^{r,2}_x)} &\lesssim \left\|\Delta \int_{t_0}^t U_\mu(t-\tau) |\nabla|^{\frac{2}{q}} F(\tau) d\tau\right\|_{L^q_t(I,L^{\overline{r},2}_x)} \\
	&\lesssim \||\nabla|^{2-\frac{2}{p}} F\|_{L^{p'}_t(I,L^{n',2}_x)}
	\end{align*}
	which is \eqref{str-est-gain}.
	
	To see \eqref{str-est-schr}, we need some preliminaries. Let $\eta \in C^\infty_0(\R^d)$ be a non-negative function supported in $\{1/2<|\xi|<2\}$ satisfying 
	\[
	\sum_{j\in \Z} \eta(2^j \xi) =1, \quad \forall \xi \in \R^d \backslash \{0\}.
	\]
	Define the Fourier multiplier
	\[
	Q_j f:= \eta(2^{-j} \cdot) \ast f, \quad j \in \Z
	\]
	and the square function
	\[
	Sf(x):= \Big( \sum_{j\in \Z} |Q_jf(x)|^2\Big)^{1/2}, \quad x\in \R^d.
	\]
	It was proved (see \cite[Lemma 2.5]{AT}) that for $1<r<\infty$ and $1\leq \rho <\infty$,
	\begin{align} \label{squa-func-est}
	\|f\|_{L^{r,\rho}_x} \sim \|Sf\|_{L^{r,\rho}_x}
	\end{align}
	and for $s\geq 0$,
	\begin{align} \label{equi-norm}
	\Big(\sum_{j\in \Z} |Q_j (-\Delta)^{s/2} f|^2\Big)^{1/2} \sim \Big(\sum_{j\in \Z} 2^{2js} |\tilde{Q}_jf|^2\Big)^{1/2},
	\end{align}
	where 
	\[
	\tilde{Q}_jf= \tilde{\eta}(2^{-j}\cdot) \ast f
	\]
	with $\tilde{\eta}\in C^\infty_0(\R^d)$ a non-negative function supported in $\{1/2<|\xi|<2\}$ and satisfying $\tilde{\eta} \eta = \eta$. 
	
	Now, for $j\in \Z$, we denote
	\[
	\Vc_{\mu,j}(t):= d_j P_j U_\mu(t),
	\]
	where $d_j f(x):= 2^{j\frac{d}{2}} f(2^jx)$ is the rescaling operator and $P_jf=\sqrt{\tilde{\eta}}(2^{-j})\ast f$.	By Plancherel's theorem, the isometry of $d_j$ on $L^2(\R^d)$ implies 
	\[
	\|\Vc_{\mu,j}(t)\|_{L^2_x \to L^2_x} \leq C, \quad \forall t\in \R.
	\]
	In addition, we have from \cite[(3.16)]{Pausader-DPDE} that 
	\begin{align*}
	\|\Vc_{\mu,j}(t)\Vc_{\mu,j}(\tau)^*\|_{L^1_x\to L^\infty_x} \leq C|t-\tau|^{-\frac{d}{2}}, \quad \forall t\ne \tau.
	\end{align*}
	Applying the result of Keel and Tao \cite{KT}, we obtain
	\begin{align*}
	\left\| \int_{t_0}^t d_j\tilde{Q}_j U_\mu(t-\tau) d_j^*F(\tau) d\tau\right\|_{L^{p_1}_t(I,L^{n_1,2}_x)} &= \left\| \int_{t_0}^t \Vc_{\mu,j}(t)\Vc_{\mu,j}(\tau)^* F(\tau) d\tau\right\|_{L^{p_1}_t(I,L^{n_1,2}_x)} \nonumber\\
	&\leq C\|F\|_{L^{p'_2}_t(I, L^{n'_2,2}_x)}. 
	\end{align*}
	As $\|d_j^*\|_{L^{n_1,2}_x \to L^{n_1,2}_x} \leq C2^{j\left(\frac{d}{n_1}-\frac{d}{2}\right)}$ and $\|d_j\|_{L^{n'_2,2}_x \to L^{n'_2,2}_x} \leq C2^{j\left(\frac{d}{2}-\frac{d}{n'_2}\right)}$, we infer that
	\begin{align*}
	\left\|\int_{t_0}^t \tilde{Q}_j U_\mu(t-\tau) F(\tau)d\tau\right\|_{L^{p_1}_t(I,L^{n_1,2}_x)} &\leq C2^{j\left(\frac{d}{n_1}-\frac{d}{2}\right)} \left\|\int_{t_0}^t d_j \tilde{Q}_j U_\mu(t-\tau) F(\tau)d\tau\right\|_{L^{p_1}_t(I,L^{n_1,2}_x)} \\
	&=C2^{j\left(\frac{d}{n_1}-\frac{d}{2}\right)} \left\|\int_{t_0}^t d_j \tilde{Q}_j U_\mu(t-\tau) d_j^* d_j F(\tau)d\tau\right\|_{L^{p_1}_t(I,L^{n_1,2}_x)} \\
	&\leq C2^{j\left(\frac{d}{n_1}-\frac{d}{2}\right)} \|d_jF\|_{L^{p_2'}_t(I,L^{n_2',2}_x)} \\
	&\leq C2^{j\left(\frac{d}{n_1}-\frac{d}{n'_2}\right)} \|F\|_{L^{p_2'}_t(I,L^{n_2',2}_x)}.
	\end{align*}
	Applying this inequality with $\tilde{\tilde{Q}}_j F$ in place of $F$, we get
	\[
	\left\| \int_{t_0}^t \tilde{Q}_j U_\mu(t-\tau) F(\tau) d\tau\right\|_{L^{p_1}_t(I,L^{n_1,2}_x)} \leq 2^{j\left(\frac{d}{n_1}-\frac{d}{n'_2}\right)}\|\tilde{\tilde{Q}}_j F\|_{L^{p'_2}_t(I, L^{n'_2,2}_x)},
	\]
	where $\tilde{\tilde{Q}}_j f :=\tilde{\tilde{\eta}} (2^{-j}\cdot) \ast f$ with $\tilde{\tilde{\eta}} \in C^\infty_0(\R^d)$ supported in $\{1/2<|\xi|<2\}$ and satisfying $\tilde{\tilde{\eta}} \tilde{\eta}=\tilde{\eta}$. By \eqref{squa-func-est}, \eqref{equi-norm}, and Minkowski's inequality, we have
	\begin{align*}
	\left\| \Delta \int_{t_0}^t U_\mu(-\tau) F(\tau) d\tau\right\|_{L^{p_1}_t(I,L^{n_1,2}_x)} &\sim \Big\|\Big( \sum_{j\in \Z} \Big|Q_j \Delta  \int_{t_0}^t U_\mu(-\tau) F(\tau) d\tau\Big|^2\Big)^{1/2} \Big\|_{L^{p_1}_t(I,L^{n_1,2}_x)} \\
	&\sim \Big\| \Big(\sum_{j\in \Z} 2^{4j} \Big|\tilde{Q}_j \int_{t_0}^t U_\mu(-\tau) F(\tau) d\tau\Big|^2 \Big)^{1/2} \Big\|_{L^{p_1}_t(I,L^{n_1,2}_x)} \\
	&\lesssim \Big(\sum_{j\in \Z} 2^{4j} \Big\| \tilde{Q}_j \int_{t_0}^t U_\mu(-\tau) F(\tau) d\tau \Big\|^2_{L^{p_1}_t(I,L^{n_1,2}_x)} \Big)^{1/2} \\
	&\lesssim \Big( \sum_{j\in \Z} 2^{4j+2j\left(\frac{d}{n_1}-\frac{d}{n'_2}\right)} \|\tilde{\tilde{Q}}_j F\|^2_{L^{p'_2}_t(I, L^{n'_2,2}_x)} \Big)^{1/2} \\
	&\lesssim \Big\| \Big(\sum_{j\in \Z} 2^{2j\(2+\frac{d}{n_1}-\frac{d}{n'_2}\)} |\tilde{\tilde{Q}}_j F|^2 \Big)^{1/2}\Big\|_{L^{p'_2}_t(I, L^{n'_2,2}_x)} \\
	&\sim \Big\| \Big(\sum_{j\in \Z} |Q_j |\nabla|^{2-\frac{2}{p_1}-\frac{2}{p_2}}F|^2 \Big)^{1/2}\Big\|_{L^{p'_2}_t(I, L^{n'_2,2}_x)}  \\
	&\sim \||\nabla|^{2-\frac{2}{p_1}-\frac{2}{p_2}} F\|_{L^{p'_2}_t(I,L^{n'_2,2}_x)}
	\end{align*}
	which is \eqref{str-est-schr}.
\end{proof}

\begin{lemma}[\bf Strichartz estimates for non biharmonic admissible pairs] \label{lem-str-est-non-adm}~\\
	Let $d\geq 1$, $\mu\geq 0$, $(q,r)\in B$ with $2<r<\infty$, and $1\leq \rho \leq \infty$. Fix $k>\frac{q}{2}$ and define $m$ by
	\begin{align} \label{cond-kmq}
	\frac{1}{k}+\frac{1}{m} =\frac{2}{q}.
	\end{align}
	Then there exists $C>0$ such that for all interval $I\subset \R$ with $0\in \overline{I}$,
	\begin{align}\label{str-est-non-adm}
	\left\|\int_0^t U_\mu(t-\tau) F(\tau) d\tau \right\|_{L^k_t(I, L^{r,\rho}_x)} \leq C\|F\|_{L^{m'}_t(I, L^{r',\rho}_x)}
	\end{align}
	for any $F\in L^{m'}(I, L^{r',\rho}(\R^d))$.
\end{lemma}

\begin{proof}
	Using \eqref{dis-est-lore}, we have for $t\in I$,
	\begin{align*}
	\left\| \int_0^t e^{i(t-\tau)\Delta} F(\tau) d\tau \right\|_{L^{r,\rho}_x} &\leq C\int_0^t |t-\tau|^{-\frac{d}{4}\(1-\frac{2}{r}\)} \|F(\tau)\|_{L^{r',\rho}_x} d\tau \\
	&=C\int_0^t |t-\tau|^{-\frac{2}{q}} \|F(\tau)\|_{L^{r',\rho}_x} d\tau \\
	&\leq C\int_{I} |t-\tau|^{-\frac{2}{q}} \|F(\tau)\|_{L^{r',\rho}_x} d\tau \\
	&= C (|\cdot|^{-\frac{2}{q}} \ast G)(t),
	\end{align*}
	where $G(t)=\mathds{1}_I(t)\|F(t)\|_{L^{r',\rho}_x}$. The result follows by the standard Hardy-Littlewood-Sobolev inequality using \eqref{cond-kmq}.
\end{proof}

\section{Nonlinear estimates}
\label{S3}
\setcounter{equation}{0}
Let $d\geq 1$, $0<b<\min\left\{\frac{d}{2},4\right\}$, $\alpha>\frac{8-2b}{d}$, and $\alpha<\frac{8-2b}{d-4}$ if $d\geq 5$. Denote
\begin{align} \label{qrkm}
\begin{aligned}
q &=\frac{8(\alpha+2)}{d\alpha+2b}, & r &= \frac{d(\alpha+2)}{d-b}, \\
k &=\frac{4\alpha(\alpha+2)}{8-2b-(d-4)\alpha}, & m&=\frac{4\alpha(\alpha+2)}{d\alpha^2+(d-4+2b)\alpha-8+2b}.
\end{aligned}
\end{align}
We readily check that $(q,r) \in B$, $2<r<\infty$, $k>\frac{q}{2}$, and $\frac{1}{k}+\frac{1}{m}=\frac{2}{q}$.

\begin{lemma} [\bf Nonlinear estimates 1] \label{lem-non-est-1} ~\\
	Let $d\geq 1$, $0<b<\min\left\{\frac{d}{2},4\right\}$, $\alpha>\frac{8-2b}{d}$, and $\alpha<\frac{8-2b}{d-4}$ if $d\geq 5$. Let $q,r,k,m$ be as in \eqref{qrkm} and $I\subset \R$ be an interval. Then we have
	\begin{align} 
	\||x|^{-b} |u|^\alpha u\|_{L^{m'}_t(I, L^{r',2}_x)} &\lesssim \|u\|^{\alpha+1}_{L^k_t(I, L^{r,2}_x)}, \label{non-est-1} \\
	\||x|^{-b} |u|^\alpha u\|_{L^{q'}_t(I, L^{r',2}_x)} &\lesssim \|u\|^\alpha_{L^k_t(I, L^{r,2}_x)} \|u\|_{L^q_t(I,L^{r,2}_x)}. \label{non-est-2}
	\end{align}
\end{lemma}

\begin{proof}
	By H\"older's inequality in Lorentz spaces, we have
	\begin{align*}
	\||x|^{-b} |u|^\alpha u\|_{L^{r',2}_x} &\leq \||x|^{-b}\|_{L^{\frac{d}{b},\infty}_x} \||u|^\alpha u\|_{L^{\overline{r},2}_x} \\
	&\lesssim \||u|^\alpha\|_{L^{\frac{r}{\alpha},\infty}_x} \|u\|_{L^{r,2}_x} \\
	&\lesssim \|u\|^\alpha_{L^{r,\infty}_x} \|u\|_{L^{r,2}_x} \\
	&\lesssim \|u\|^{\alpha+1}_{L^{r,2}_x},
	\end{align*}
	where $1<\overline{r}<\infty$ is such that
	\[
	\frac{1}{r'}=\frac{b}{d}+\frac{1}{\overline{r}}, \quad \frac{1}{\overline{r}} = \frac{\alpha+1}{r} \quad \text{or} \quad \overline{r} = \frac{d(\alpha+2)}{(d-b)(\alpha+1)}.
	\]
	As $\frac{1}{m'}=\frac{\alpha+1}{k}$, the standard H\"older inequality in time yields
	\[
	\||x|^{-b} |u|^\alpha u\|_{L^{m'}_t L^{r',2}_x} \lesssim \|u\|^{\alpha+1}_{L^k_t, L^{r,2}_x}
	\]
	which is \eqref{non-est-1}. Estimate \eqref{non-est-2} is treated similarly using the fact that $\frac{1}{q'}=\frac{\alpha}{k}+\frac{1}{q}$.
\end{proof}

\begin{lemma}[\bf Nonlinear estimates 2] \label{lem-non-est-2} ~\\
	Let $d\geq 3$, $0<b<\min\left\{\frac{d}{2},4\right\}$, $\alpha>\frac{8-2b}{d}$, and $\alpha<\frac{8-2b}{d-4}$ if $d\geq 5$. Let $k, r$ be as in \eqref{qrkm} and $I\subset \R$ be an interval. Then there exist $(p,n)\in S$ and $(\beta, \nu)\in B$ with $n,\nu<\infty$ such that
	\begin{align}
	\||\nabla|^{2-\frac{2}{p}}[|x|^{-b}|u|^\alpha u]\|_{L^{p'}_t(I,L^{n',2}_x)} \lesssim \|u\|^\alpha_{L^k_t(I, L^{r,2}_x)} \|\Delta u\|_{L^\beta_t(I,L^{\nu,2}_x)}. \label{non-est-3}
	\end{align}
\end{lemma}

\begin{proof}
	We consider two cases:
	
	{\bf Case 1.} Let $d\geq 3$, $0<b<\min\left\{\frac{d}{2},4\right\}$, $\alpha>\frac{8-2b}{d}$, and $\alpha<\frac{8-2b}{d-4}$ if $d\geq 5$. Assume in addition that $b> 1$ if $d=3$. In this case, we take $(p,n)=\(2,\frac{2d}{d-2}\) \in S$, hence $2-\frac{2}{p} =1$. We introduce the following exponents:
	\begin{align*}
	\delta:&= \frac{2d}{d+2-2b}, & \sigma:&=\frac{2d(\alpha+2)}{2d+4-2b-(d-2)\alpha}, \\
	\beta:&=\frac{4(\alpha+2)}{(d-2)\alpha-4+2b}, & \nu:&=\frac{2d(\alpha+2)}{2(d+4-2b)-(d-4)\alpha}.
	\end{align*}
	Since $\alpha>\frac{8-2b}{d}$, the condition $b>1$ when $d=3$ is needed to ensure $(d-2)\alpha -4+2b>0$, hence $\beta$ has positive value. One readily check that $(\beta,\nu) \in B$, $1<\delta<d$, and
	\[
	\frac{d+2}{2d}=\frac{b}{d}+\frac{1}{\delta}, \quad \frac{1}{\delta}=\frac{\alpha}{r}+\frac{1}{\sigma}, \quad \frac{1}{\sigma}=\frac{1}{\nu}-\frac{1}{d}, \quad \frac{1}{2} =\frac{\alpha}{k} +\frac{1}{\beta}.
	\]
	We have
	\[
	\|\nabla[|x|^{-b}|u|^\alpha u]\|_{L^2_t L^{\frac{2d}{d+2},2}_x} \leq \||x|^{-b} \nabla(|u|^\alpha u)\|_{L^2_t L^{\frac{2d}{d+2},2}_x} + \|\nabla (|x|^{-b}) |u|^\alpha u\|_{L^2_t L^{\frac{2d}{d+2},2}_x} = (\text{A}) + (\text{B}).
	\]
	By H\"older's inequality, the chain rule, and Sobolev embedding in Lorentz spaces, we have
	\begin{align*}
	\||x|^{-b} \nabla(|u|^\alpha u)\|_{L^{\frac{2d}{d+2},2}_x} &\leq \||x|^{-b}\|_{L^{\frac{d}{b},\infty}_x} \|\nabla (|u|^\alpha u)\|_{L^{\delta,2}_x} \\
	&\lesssim \||u|^\alpha\|_{L^{\frac{r}{\alpha},\infty}_x} \|\nabla u\|_{L^{\sigma,2}_x} \\
	&\lesssim \|u\|^\alpha_{L^{r,\infty}_x} \|\Delta u\|_{L^{\nu,2}_x} \\
	&\lesssim \|u\|^\alpha_{L^{r,2}_x} \|\Delta u\|_{L^{\nu,2}_x},
	\end{align*}
	which, by H\"older's inequality in time, yields
	\[
	(\text{A}) \lesssim \|u\|^\alpha_{L^k_tL^{r,2}_x} \|\Delta u\|_{L^\beta_t L^{\nu,2}_x}.
	\]
	We next write 
	\[
	\nabla (|x|^{-b}) |u|^\alpha u \sim |x|^{-b} |x|^{-1}(|u|^\alpha u)
	\]
	and use Hardy's inequality in Lorentz spaces with $1<\delta<d$ to get
	\[
	\|\nabla (|x|^{-b}) |u|^\alpha u\|_{L^{\frac{2d}{d+2},2}_x} \leq \||x|^{-b}\|_{L^{\frac{d}{b},\infty}_x} \||x|^{-1}(|u|^\alpha u)\|_{L^{\delta, 2}_x} \lesssim \|\nabla (|u|^\alpha u)\|_{L^{\delta,2}_x}.
	\]
	Repeating the same estimates as above, we prove as well that
	\[
	(\text{B}) \lesssim \|u\|^\alpha_{L^k_tL^{r,2}_x} \|\Delta u\|_{L^\beta_t L^{\nu,2}_x}.
	\]
	This shows \eqref{non-est-3} in this case.
	
	{\bf Case 2.} Let $d=3$, $0<b\leq 1$, and $\alpha>\frac{8-2b}{3}$. We take $(p,n) = \(\frac{7-b}{3}, \frac{14-2b}{3-b}\)\in S$, hence $2-\frac{2}{p}= \frac{4-b}{7-b}$. We define 
	\begin{align*}
	\delta:&=\frac{6(7-b)}{33-17b+2b^2}, & \sigma:&=\frac{6(7-b)(\alpha+2)}{66-34b+4b^2-(9-3b)\alpha}, \\
	\beta:&=\frac{4(7-b)(\alpha+2)}{(9-3b)\alpha+14b-2b^2-24}, & \nu:&=\frac{6(7-b)(\alpha+2)}{90-34b+4b^2+(3+3b)\alpha}, \\
	\theta:&=\frac{6(7-b)}{25-15b+2b^2}.& &
	\end{align*}
	We have the following properties: $(\beta,\nu)\in B$, $1<\delta<\frac{3}{2-\frac{2}{p}}$, and
	\[
	\frac{1}{n'}=\frac{b}{3}+\frac{1}{\delta}=\frac{b+2-\frac{2}{p}}{3}+\frac{1}{\theta}, \quad \frac{1}{\delta}=\frac{\alpha}{r}+\frac{1}{\sigma}, \quad \frac{1}{\sigma}=\frac{1}{\nu}-\frac{2}{3p}, \quad \frac{1}{p'}=\frac{\alpha}{k}+\frac{1}{\beta}.
	\]
	By the product rule, Sobolev embedding, and the chain rule in Lorentz spaces, we estimate 
	\begin{align*}
	\||\nabla|^{2-\frac{2}{p}} [|x|^{-b}|u|^\alpha u]\|_{L^{n',2}_x} &\lesssim \||x|^{-b}\|_{L^{\frac{3}{b},\infty}_x} \||\nabla|^{2-\frac{2}{p}} (|u|^\alpha u)\|_{L^{\delta,2}_x} \\ 
	&\quad + \||\nabla|^{2-\frac{2}{p}} (|x|^{-b})\|_{L^{\frac{3}{b+2-\frac{2}{p}},\infty}_x} \||u|^\alpha u\|_{L^{\theta,2}_x} \\
	&\lesssim \||\nabla|^{2-\frac{2}{p}} (|u|^\alpha u)\|_{L^{\delta,2}_x} \\
	&\lesssim \||u|^\alpha\|_{L^{\frac{r}{\alpha},\infty}_x} \||\nabla|^{2-\frac{2}{p}} u\|_{L^{\sigma,2}_x} \\
	&\lesssim \|u\|^\alpha_{L^{r,\infty}_x} \|\Delta u\|_{L^{\nu,2}_x} \\
	&\lesssim \|u\|^\alpha_{L^{r,2}_x} \|\Delta u\|_{L^{\nu,2}_x},
	\end{align*}
	H\"older's inequality in time yields
	\[
	\||\nabla|^{2-\frac{2}{p}} [|x|^{-b}|u|^\alpha u]\|_{L^{p'}_tL^{n',2}_x} \lesssim \|u\|^\alpha_{L^k_t L^{r,2}_x} \|\Delta u\|_{L^\beta_t L^{\nu,2}_x}.
	\]
	This is the desired estimate.
\end{proof}


\begin{lemma}[\bf Nonlinear estimates 3] \label{lem-non-est-3} ~\\
	Let $d\geq 5$, $0<b<\min\left\{\frac{d}{2},4\right\}$, $\frac{8-2b}{d}<\alpha<\frac{8-2b}{d-4}$, and $\frac{2d}{d+8}\leq n \leq \frac{2d}{d+4}$. Then we have
	\begin{align} \label{non-est-4}
	\||x|^{-b} |u|^\alpha u\|_{L^{n,2}_x} \lesssim \|u\|^{\alpha+1}_{H^2_x}.
	\end{align}
	In addition, we have for $d=4$, $0<b<\frac{d}{2}$, and $\alpha>\frac{8-2b}{d}$, 
	\begin{align} \label{non-est-5}
	\||x|^{-b} |u|^\alpha u\|_{L^{\frac{1}{1-\theta},2}_x} \lesssim \|u\|^{\frac{\alpha}{2}}_{H^2_x} \(\int_{\R^d} |x|^{-b} |u(x)|^{\alpha+2} dx\)^{\frac{1}{2}}
	\end{align}
	provided that $\theta>0$ is taken sufficiently small. Furthermore, we have for $d=3$, $0<b<\frac{d}{2}$, and $\alpha>\frac{8-2b}{d}$,
	\begin{align} \label{non-est-6}
	\||x|^{-b} |u|^\alpha u\|_{L^{\frac{1}{1-\theta},2}_x} \lesssim \|u\|^{\frac{3\alpha-2}{8}}_{H^2_x} \(\int_{\R^d} |x|^{-b} |u(x)|^{\alpha+2} dx\)^{\frac{5}{8}}
	\end{align}
	provided that $\theta>0$ is taken sufficiently small.
\end{lemma}

\begin{proof}
	Let $d\geq 5$. We estimate
	\begin{align*}
	\||x|^{-b}|u|^\alpha u\|_{L^{n,2}_x} &\leq \||x|^{-b}\|_{L^{\frac{d}{b},\infty}_x} \||u|^\alpha u\|_{L^{\delta, 2}_x} \\
	&\lesssim \|u\|^{\alpha+1}_{L^{\delta(\alpha+1), 2(\alpha+1)}_x} \\
	&\lesssim \|\scal{\Delta} u\|^{\alpha+1}_{L^{2,2(\alpha+1)}_x} \\
	&\lesssim \|u\|^{\alpha+1}_{H^2_x}
	\end{align*}
	provided that $\frac{1}{n}=\frac{b}{d}+\frac{1}{\delta}$ and 
	\begin{align}\label{cond-delta}
	2\leq \delta(\alpha+1) \leq \frac{2d}{d-4}.
	\end{align}
	Here we have used the Sobolev embedding $W^2L^{2,\rho}(\R^d)\subset L^{n, \rho}(\R^d)$ for $2\leq n\leq \frac{2d}{d-4}$, and $1\leq \rho \leq \infty$. Condition \eqref{cond-delta} is equivalent to 
	\[
	2\leq \frac{dn(\alpha+1)}{d-bn} \leq \frac{2d}{d-4}.
	\]
	As $\frac{8-2b}{d}<\alpha<\frac{8-2b}{d-4}$, and $\frac{2d}{d+8}\leq n \leq \frac{2d}{d+4}$, the above condition is satisfied and \eqref{non-est-4} follows.
	
	We now consider the case $d=4$. We estimate
	\begin{align*}
	\||x|^{-b} |u|^\alpha u\|_{L^{\frac{1}{1-\theta},2}_x} &\leq \||x|^{-\frac{b}{2}}\|_{L^{\frac{2d}{b},\infty}_x} \||u|^{\frac{\alpha}{2}}\|_{L^{\delta, \infty}_x} \||x|^{-\frac{b}{2}} |u|^{\frac{\alpha+2}{2}} \|_{L^{2,2}_x} \\
	&\lesssim \|u\|^{\frac{\alpha}{2}}_{L^{\frac{\delta\alpha}{2},\infty}_x} \(\int_{\R^d} |x|^{-b} |u(x)|^{\alpha+2} dx\)^{\frac{1}{2}} \\
	&\lesssim \|u\|^{\frac{\alpha}{2}}_{H^2_x} \(\int_{\R^d} |x|^{-b} |u(x)|^{\alpha+2} dx\)^{\frac{1}{2}}
	\end{align*}
	provided that 
	\[
	1-\theta=\frac{b}{2d} + \frac{1}{\delta} +\frac{1}{2}, \quad \frac{\delta \alpha}{2}>2.
	\]
	Here we have used the embedding $H^2(\R^4) \subset L^{\frac{\delta \alpha}{2}}(\R^4) \subset L^{\frac{\delta\alpha}{2},\infty}(\R^4)$ for $2<\frac{\delta\alpha}{2}<\infty$.	In particular, we need
	\[
	\frac{d\alpha}{d-b-2d\theta} >2 \text{ or } d\alpha-2d+2b +4d\theta>0
	\]
	which is satisfied due to $d= 4$, $\alpha>\frac{8-2b}{d}$, and $\theta>0$ sufficiently small. 
	
	Finally, for $d=3$, we have
	\begin{align*}
	\||x|^{-b} |u|^\alpha u\|_{L^{\frac{1}{1-\theta},2}_x} &\leq \||x|^{-\frac{3b}{8}}\|_{L^{\frac{8d}{3b},\infty}_x} \||u|^{\frac{3\alpha-2}{8}}\|_{L^{\delta, \infty}_x} \||x|^{-\frac{5b}{8}} |u|^{\frac{5(\alpha+2)}{8}} \|_{L^{\frac{8}{5},2}_x} \\
	&\lesssim \|u\|^{\frac{3\alpha-2}{8}}_{L^{\frac{\delta(3\alpha-2)}{8},\infty}_x} \(\int_{\R^d} |x|^{-b} |u(x)|^{\alpha+2} dx\)^{\frac{5}{8}} \\
	&\lesssim \|u\|^{\frac{3\alpha-2}{8}}_{H^2_x} \(\int_{\R^d} |x|^{-b} |u(x)|^{\alpha+2} dx\)^{\frac{5}{8}}
	\end{align*}
	provided that 
	\[
	1-\theta=\frac{3b}{8d} + \frac{1}{\delta} +\frac{5}{8}, \quad \frac{\delta (3\alpha-2)}{8}>2.
	\]
	Here we have used the embedding $H^2(\R^3) \subset L^{\frac{\delta (3\alpha-2)}{8}}(\R^3) \subset L^{\frac{\delta(3\alpha-2)}{8},\infty}(\R^3)$ for $2<\frac{\delta(3\alpha-2)}{8}<\infty$ and $L^{\frac{8}{5}}(\R^3)\subset L^{\frac{8}{5},2}(\R^3)$. In particular, we need
	\[
	\frac{d(3\alpha-2)}{3d-3b-8d\theta} >2 \text{ or } 3d\alpha-8d+6b +16d\theta>0
	\]
	which is satisfied as $d= 3$, $\alpha>\frac{8-2b}{d}$, and $\theta>0$ sufficiently small.
\end{proof}

A direct consequence of the above nonlinear estimates is the following scattering condition. 

\begin{lemma}[\bf Scattering condition] \label{lem-scat-cond} ~\\
	Let $d\geq 3$, $\mu\geq 0$, $0<b<\min\left\{\frac{d}{2},4\right\}$, $\alpha>\frac{8-2b}{d}$, and $\alpha<\frac{8-2b}{d-4}$ if $d\geq 5$. Suppose that $u$ is a global-in-time solution to \eqref{IBNLS} satisfying
	\[
	\sup_{t\in \R} \|u(t)\|_{H^2_x} \leq E
	\]
	for some constant $E>0$. If 
	\begin{align}\label{boun-LkLr2}
	\|u\|_{L^k_t(\R, L^{r,2}_x)} <\infty,
	\end{align}
	where $k,r$ is as in \eqref{qrkm}, then the solution scatters in $H^2(\R^d)$ in both directions.
\end{lemma}

\begin{proof}
	Let $\vareps>0$ to be chosen shortly. Thanks to \eqref{boun-LkLr2}, there exist an integer $J=J(\vareps)>0$ and disjoint intervals $I_j, j=1, \cdots, J$ such that
	\begin{align} \label{vareps-Ij}
	\|u\|_{L^k_t(I_j, L^{r,2}_x)} <\vareps,\quad j=1,\cdots, J.
	\end{align}
	For $t\in I_j=[t_j, t_{j+1}]$, we have the Duhamel formula
	\[
	u(t) = U_\mu(t-t_j) u(t_j) + i \int_{t_j}^t U_\mu(t-\tau) (|x|^{-b} |u|^\alpha u)(\tau) d\tau.
	\]
	By Strichartz estimates \eqref{str-est-inhomo}, \eqref{str-est-gain} and nonlinear estimates \eqref{non-est-2}, \eqref{non-est-3}, we have
	\begin{align*}
	\|\scal{\Delta} u\|_{L^q_t(I_j,L^{r,2}_x)} &\leq \|\scal{\Delta} U_\mu(t-t_j)u(t_j)\|_{L^q_t(I_j,L^{r,2}_x)} \\
	&\quad + \left\|\int_{t_j}^t U_\mu(t-\tau)  (|x|^{-b} |u|^\alpha u) (\tau) d\tau\right\|_{L^q_t(I_j,L^{r,2}_x)} \\
	&\quad + \left\|\int_{t_j}^t U_\mu(t-\tau)  \Delta [|x|^{-b} |u|^\alpha u] (\tau) d\tau\right\|_{L^q_t(I_j,L^{r,2}_x)} \\
	&\lesssim \|u(t_j)\|_{H^2_x} + \||x|^{-b} |u|^\alpha u\|_{L^{q'}_t(I_j,L^{r',2}_x)} + \||\nabla|^{2-\frac{2}{p}}[|x|^{-b} |u|^\alpha u]\|_{L^{p'}_t (I_j,L^{n',2}_x)}\\
	&\lesssim \|u(t_j)\|_{H^2_x} + \|u\|^\alpha_{L^k_t(I_j,L^{r,2}_x)} \|u\|_{L^q_t(I_j,L^{r,2}_x)} + \|u\|^\alpha_{L^k_t(I_j,L^{r,2}_x)} \|\Delta u\|_{L^\beta_t (I_j,L^{\nu,2}_x)}\\
	&\lesssim E + \vareps^\alpha \(\|\scal{\Delta} u\|_{L^q_t(I_j,L^{r,2}_x)} + \|\scal{\Delta} u\|_{L^\beta_t (I_j,L^{\nu,2}_x)}\),
	\end{align*}
	where $p,n, \beta, \nu$ are as in Lemma \ref{lem-non-est-2}. A similar estimate goes for $\|\scal{\Delta} u\|_{L^\beta_t(I_j,L^{\nu,2}_x)}$ and we get
	\[
	\|\scal{\Delta} u\|_{L^q_t(I_j,L^{r,2}_x)}+\|\scal{\Delta} u\|_{L^\beta_t(I_j,L^{\nu,2}_x)} \lesssim E + \vareps^\alpha \(\|\scal{\Delta} u\|_{L^q_t(I_j,L^{r,2}_x)} + \|\scal{\Delta} u\|_{L^\beta_t (I_j,L^{\nu,2}_x)}\).
	\]
	Taking $\vareps>0$ small, there exists $C>0$ such that
	\begin{align*}
	\|\scal{\Delta} u\|_{L^q_t(I_j,L^{r,2}_x)} + \|\scal{\Delta} u\|_{L^\beta_t(I_j,L^{\nu,2}_x)} \leq CE, \quad \forall j=1,\cdots, J.
	\end{align*}
	Summing over all sub-intervals $I_j$, we obtain
	\begin{align}  \label{glob-norm}
	\|\scal{\Delta} u\|_{L^q_t(\R,L^{r,2}_x)} + \|\scal{\Delta} u\|_{L^\beta_t(\R,L^{\nu,2}_x)} <\infty.
	\end{align}
	
	We now prove the scattering for positive times (the one for negative times is treated similarly). Let $t_2>t_1>0$. We have
	\begin{align*}
	\|U_\mu(-t_2) u(t_2)&-U_\mu(-t_1) u(t_1)\|_{H^2_x} \\
	&\leq \left\|\int_{t_1}^{t_2} U_\mu(-\tau) (|x|^{-b}|u|^\alpha u)(\tau) d\tau\right\|_{L^2_x} \\
	&\quad + \left\|\int_{t_1}^{t_2} U_\mu(-\tau) \Delta [|x|^{-b}|u|^\alpha u](\tau) d\tau\right\|_{L^2_x} \\
	&\lesssim \|u\|^\alpha_{L^k_t((t_1,t_2),L^{r,2}_x)} \left(\|\scal{\Delta} u\|_{L^q_t((t_1,t_2),L^{r,2}_x)}+ \|\scal{\Delta} u\|_{L^\beta_t((t_1,t_2),L^{\nu,2}_x)}\)
	\end{align*}
	Thanks to \eqref{boun-LkLr2} and \eqref{glob-norm}, we see that
	\[
	\|U_\mu(-t_2) u(t_2)-U_\mu(-t_1) u(t_1)\|_{H^2_x}  \to 0 \text{ as } t_1, t_2 \to +\infty.
	\]
	This shows that $(U_\mu(-t) u(t))_{t\to +\infty}$ is a Cauchy sequence in $H^2(\R^d)$. Thus the limit 
	\[
	u_+ := u_0 + i \int_t^{+\infty} U_\mu(-\tau) (|x|^{-b} |u|^\alpha u)(\tau) d\tau
	\]
	exists in $H^2(\R^d)$. Estimating as above, we prove that
	\[
	\|u(t)-U_\mu(t) u_+\|_{H^2_x} \to 0 \text{ as } t\to +\infty.
	\]
	The proof is complete.
\end{proof}

\section{Variational analysis}
\label{S4}
\setcounter{equation}{0}
In this section, we study the minimization problem \eqref{m-mu-omega} and we have the following existence of minimizers for $m_{\mu,\omega}$.

\begin{proposition}[\bf Existence of minimizers] \label{prop-mini-prob} ~\\
	Let $d\geq 1$, $\mu\geq 0$, $0<b<\min\{d,4\}$, $\alpha>\frac{8-2b}{d}$, $\alpha<\frac{8-2b}{d-4}$ if $d\geq 5$, and $\omega>0$. Then $m_{\mu,\omega}>0$ and there exists at least a minimizer for $m_{\mu,\omega}$ which solves the elliptic equation \eqref{elli-equa}. In addition,
	\begin{align} \label{chara-m-omega}
	m_{\mu,\omega}= \inf \left\{S_{\mu,\omega}(f) : f \in H^2(\R^d) \backslash \{0\}, f \text{ solves } \eqref{elli-equa}\right\}.
	\end{align}
\end{proposition}

Before giving the proof of Proposition \ref{prop-mini-prob}, we have the following observation.

\begin{observation} \label{obse-1}
	Let $f\in H^2(\R^d)\backslash \{0\}$. Then there exists a unique $\lambda_0>0$ such that
	\[
	G_\mu(f_\lambda) \left\{
	\begin{array}{cl}
	>0 &\text{if } 0<\lambda<\lambda_0, \\
	=0 &\text{if } \lambda=\lambda_0, \\
	<0 &\text{if } \lambda>\lambda_0,
	\end{array}
	\right.
	\]
	where 
	\[
	f_\lambda(x):= \lambda^{\frac{d}{2}} f(\lambda x), \quad \lambda>0
	\]
	is the mass-critical scaling.
\end{observation}

\begin{proof}[Proof of Observation \ref{obse-1}]
	We have
	\begin{align*}
	G_\mu(f_\lambda)&= 2\lambda^4 \|\Delta f\|^2_{L^2} + \mu \lambda^2 \|\nabla f\|^2_{L^2} -\frac{d\alpha+2b}{2(\alpha+2)} \lambda^{\frac{d\alpha+2b}{2}} \int |x|^{-b} |f(x)|^{\alpha+2} dx \\
	&= \lambda^4 \phi(\lambda)
	\end{align*}
	where
	\[
	\phi(\lambda) = a +b \lambda^{-2} - c\lambda^\delta
	\]
	with $a,c>0$, $b\geq 0$, and $\delta>0$. As $\phi(\lambda)$ is strictly decreasing on $(0,+\infty)$, $\lim_{\lambda \to 0} \phi(\lambda) \geq a>0$, and $\lim_{\lambda \to +\infty} \phi(\lambda) =-\infty$, there exists a unique $\lambda_0>0$ such that $\phi(\lambda_0)=0$, $\phi(\lambda)>0$ for $0<\lambda<\lambda_0$, and $\phi(\lambda)<0$ for $\lambda>\lambda_0$. This proves the observation. 
\end{proof}


\begin{proof}[Proof of Proposition \ref{prop-mini-prob}]
	Let $(f_n)_n$ be a minimizing sequence for $m_{\mu,\omega}$, i.e., $S_{\mu,\omega}(f_n) \to mu_{\mu,\omega}$ as $n\to \infty$ and $G_\mu(f_n)=0$ for all $n\geq 1$. 
	
	We first show that $m_{\mu,\omega}>0$. To see this, we observe that
	\begin{align}
	m_{\mu,\omega} \longleftarrow S_{\mu,\omega}(f_n) &=S_{\mu,\omega}(f_n) - \frac{2}{d\alpha+2b} G_\mu(f_n) \nonumber\\
	&=\frac{d\alpha+2b-8}{2(d\alpha+2b)} \|\Delta f_n\|^2_{L^2} + \frac{(d\alpha+2b-4)\mu}{2(d\alpha+2b)} \|\nabla f_n\|^2_{L^2} + \frac{\omega}{2} \|f_n\|^2_{L^2}. \label{est-m-omega}
	\end{align}
	As $\mu\geq 0$, $\omega>0$, and $\alpha>\frac{8-2b}{d}$, we infer that $(f_n)_n$ is a bounded sequence in $H^2(\R^d)$. As $G_\mu(f_n)=0$, we use \eqref{GN-ineq} to get
	\begin{align*}
	2\|\Delta f_n\|^2_{L^2} + \mu \|\nabla f_n\|^2_{L^2} &= \frac{d\alpha+2b}{2(\alpha+2)} \int |x|^{-b} |f_n(x)|^{\alpha+2} dx \\
	&\leq C \|\Delta f_n\|^{\frac{d\alpha+2b}{4}}_{L^2} \|f_n\|^{\frac{8-2b-(d-4)\alpha}{4}}_{L^2} \\
	&\leq C\(2\|\Delta f_n\|^2_{L^2} + \mu\|\nabla f_n\|^2_{L^2}\)^{\frac{d\alpha+2b}{8}}.
	\end{align*}
	Since $d\alpha + 2b>8$, there exists $C>0$ such that for all $n\geq 1$,
	\[
	2\|\Delta f_n\|^2_{L^2} +\mu \|\nabla f_n\|^2_{L^2} \geq C
	\]
	which together with \eqref{est-m-omega} imply $m_{\mu,\omega}>0$. We also have
	\begin{align} \label{lowe-boun}
	\int |x|^{-b} |f_n(x)|^{\alpha+2} dx \geq C.
	\end{align}
	
	We next show the existence of a minimizer for $m_{\mu, \omega}$. Since $(f_n)_n$ is a bounded sequence in $H^2(\R^d)$, there exists $f\in H^2(\R^d)$ such that up to a subsequence,
	\begin{itemize}[leftmargin=6mm]
		\item $f_n \rightharpoonup f$ weakly in $H^2(\R^d)$;
		\item $f_n \to f$ a.e. in $\R^d$;
		\item $f_n \to f$ strongly in $L^r_{\loc}(\R^d)$ for all $r\geq 1$ and $r<\frac{2d}{d-4}$ if $d\geq 5$.
	\end{itemize}
	We claim that 
	\[
	\int |x|^{-b} |f_n(x)|^{\alpha+2} dx \to \int |x|^{-b} |f(x)|^{\alpha+2} dx \text{ as } n \to \infty.
	\]
	This combined with \eqref{lowe-boun} imply $f\ne 0$. To see the claim, we take $\vareps>0$. For $R>0$ to be chosen later depending on $\vareps$, we write
	\begin{align*}
	\Big|\int |x|^{-b} |f_n(x)|^{\alpha+2} dx &- \int |x|^{-b} |f(x)|^{\alpha+2} dx \Big| \\
	&\leq \int_{|x|\leq R} |x|^{-b}||f_n(x)|^{\alpha+2} - |f(x)|^{\alpha+2}| dx \\
	&\quad + \int_{|x|> R} |x|^{-b}||f_n(x)|^{\alpha+2} - |f(x)|^{\alpha+2}| dx\\
	&=: (\text{I}) + (\text{II}).
	\end{align*}
	We estimate
	\begin{align*}
	(\text{II}) &\leq R^{-b} \(\|f_n\|^{\alpha+2}_{L^{\alpha+2}} + \|f\|^{\alpha+2}_{L^{\alpha+2}}\) \\
	&\leq R^{-b} \(\|f_n\|^{\alpha+2}_{H^2} + \|f\|^{\alpha+2}_{H^2}\) \\
	&\leq C R^{-b} <\frac{\vareps}{2}
	\end{align*}
	provided that $R = a \vareps^{-\frac{1}{b}}$ with $a>0$ sufficiently large. Here we have $\alpha+2<\frac{2d}{d-4}$ if $d\geq 5$ (as $\alpha<\frac{8-2b}{d-4}$) which ensures $H^2(\R^d) \subset L^{\alpha+2}(\R^d)$. 
	
	To estimate $(\text{I})$, we pick $\gamma=\frac{d}{b}-\delta$ and $r=\frac{d-\delta b}{d-b-\delta b}$ with some $\delta>0$ small and estimate
	\begin{align*}
	(\text{I}) &\leq \||x|^{-b}\|_{L^\gamma(|x|\leq R)} \||f_n|^{\alpha+2} -|f|^{\alpha+2}\|_{L^r(|x|\leq R)} \\
	&\leq C R^{\delta b} \(\|f_n\|^{\alpha+1}_{L^{(\alpha+2)r}} +\|f\|^{\alpha+1}_{L^{(\alpha+2)r}}\) \|f_n-f\|_{L^{(\alpha+2)r}(|x|\leq R)} \\
	&\leq CR^{\delta b} \(\|f_n\|^{\alpha+1}_{H^2} +\|f\|^{\alpha+1}_{H^2}\) \|f_n-f\|_{L^{(\alpha+2)r}(|x|\leq R)} \\
	&\leq CR^{\delta b} \|f_n-f\|_{L^{(\alpha+2)r}(|x|\leq R)},
	\end{align*}
	where $(\alpha+2)r >2$ and $(\alpha+2)r <\frac{2d}{d-4}$ if $d\geq 5$. The later comes from the assumption $\alpha<\frac{8-2b}{d-4}$ and $\delta>0$ small. As $f_n \to f$ strongly in $L^{(\alpha+2)r}(|x|\leq R)$, there exists $n_0\in \N$ large such that for all $n\geq n_0$, $(\text{I})<\frac{\vareps}{2}$. Collecting the above estimates, we prove the claim.
	
	By the weak convergence in $H^2(\R^d)$ and the above claim, we have
	\[
	G_\mu(f) \leq \liminf_{n\to \infty} G_\mu(f_n) =0.
	\]
	By Observation \eqref{obse-1}, there exists $\lambda_0\in (0,1]$ such that $G_\mu(f_{\lambda_0})=0$. Thus
	\begin{align*}
	m_{\mu,\omega}\leq S_{\mu,\omega}(f_{\lambda_0}) &=S_{\mu,\omega}(f_{\lambda_0}) - \frac{2}{d\alpha+2b} G_\mu(f_{\lambda_0}) \\
	&=\frac{d\alpha+2b-8}{2(d\alpha+2b)} \|\Delta f_{\lambda_0}\|^2_{L^2} + \frac{(d\alpha+2b-4)\mu}{2(d\alpha+2b)} \|\nabla f_{\lambda_0}\|^2_{L^2} + \frac{\omega}{2} \|f_{\lambda_0}\|^2_{L^2} \\
	&\leq \frac{d\alpha+2b-8}{2(d\alpha+2b)} \|\Delta f\|^2_{L^2} + \frac{(d\alpha+2b-4)\mu}{2(d\alpha+2b)} \|\nabla f\|^2_{L^2} + \frac{\omega}{2} \|f\|^2_{L^2} \\
	&\leq \liminf_{n \to \infty} \frac{d\alpha+2b-8}{2(d\alpha+2b)} \|\Delta f_n\|^2_{L^2} + \frac{(d\alpha+2b-4)\mu}{2(d\alpha+2b)} \|\nabla f_n\|^2_{L^2} + \frac{\omega}{2} \|f_n\|^2_{L^2} \\
	&= \liminf_{n\to \infty} S_{\mu,\omega}(f_n) - \frac{2}{d\alpha+2b} G_\mu(f_n) \\
	&=m_{\mu,\omega},
	\end{align*}
	where we have used the fact $\lambda_0\leq 1$ to get the third line. This shows that $\lambda_0=1$ or $G_\mu(f)=0$ and $S_{\mu,\omega} = m_{\mu,\omega}$. In particular, $f$ is a minimizer for $m_{\mu,\omega}$. 
	
	We now show that $f$ is a solution to \eqref{elli-equa}. As $f$ is a minimizer for $m_{\mu,\omega}$, there exists a Lagrange multiplier $\nu\in \R$ such that $S'_{\mu,\omega}(f) = \nu G'_\mu(f)$. In particular, we have
	\begin{align}\label{lagrange}
	(1-4\nu)\Delta^2f -\mu(1+2\nu)\Delta f +\omega f -\(1-\frac{d\alpha+2b}{2}\nu\)|x|^{-b}|f|^\alpha f = 0
	\end{align}
	in $H^{-2}(\R^d)$. Multiplying both sides of \eqref{lagrange} with $\overline{f}$ and integrating over $\R^d$, we get
	\begin{align} \label{equ-1}
	(1-4\nu)\|\Delta f\|^2_{L^2} + \mu(1+2\nu) \|\nabla f\|^2_{L^2} + \omega \|f\|^2_{L^2} - \(1-\frac{d\alpha+2b}{2}\nu\) \int|x|^{-b} |f|^{\alpha+2} dx =0.
	\end{align}
	Multiplying both sides of \eqref{lagrange} with $x\cdot \nabla \overline{f}$, integrating over $\R^d$, and taking the real part, we have
	\begin{align} \label{equ-2}
	\begin{aligned}
	(1-4\nu)\(2-\frac{d}{2}\) \|\Delta f\|^2_{L^2} &+\mu(1+2\nu)\(1-\frac{d}{2}\) \|\nabla f\|^2_{L^2} - \omega \frac{d}{2} \|f\|^2_{L^2} \\
	&-\(1-\frac{d\alpha+2b}{2}\nu\)\frac{b-d}{\alpha+2}\int |x|^{-b}|f|^{\alpha+2} dx =0.
	\end{aligned}
	\end{align}
	Here we have used the following identity (see \cite[(3.12)]{Dinh-JDDE})
	\[
	\rea \int (-\Delta)^\gamma f x\cdot \nabla \overline{f} dx = \(\gamma-\frac{d}{2}\) \|(-\Delta)^{\gamma/2} f\|^2_{L^2}
	\]
	and also
	\[
	\rea \int |x|^{-b} |f|^\alpha f x \cdot \nabla \overline{f} dx = \frac{b-d}{\alpha+2} \int |x|^{-b}|f|^{\alpha+2} dx.
	\]
	Combining \eqref{equ-1}, \eqref{equ-2}, and $G_\mu(f)=0$, we have
	\[
	(d\alpha+2b-8)\nu \|\Delta f\|^2_{L^2} + \frac{\mu}{2} (d\alpha+2b+4)\nu \|\nabla f\|^2_{L^2}=0,
	\]
	hence $\nu=0$ due to $d\alpha+2b>8$ and $\mu\geq 0$. In particular, $f$ is a solution to \eqref{elli-equa}. 
	
	Finally, we prove \eqref{chara-m-omega}. Denote
	\[
	n_{\mu,\omega}:= \inf \left\{S_{\mu,\omega}(f): f\in H^2(\R^d)\backslash \{0\}, f \text{ solves } \eqref{elli-equa}\right\}.
	\]
	Let $f$ be a minimizer for $m_{\mu,\omega}$. In particular, $f$ is a solution to \eqref{elli-equa}, hence $n_{\mu,\omega}\leq S_{\mu,\omega}(f)=m_{\mu,\omega}$. Let $f$ be a solution to \eqref{elli-equa}. Arguing as in \eqref{equ-1} and \eqref{equ-2}, we deduce that $G_\mu(f)=0$. Thus $m_{\mu,\omega}\leq S_{\mu,\omega}(f)$ which, by taking the infimum over all non-trivial solutions to \eqref{elli-equa}, yields $m_{\mu,\omega}\leq n_{\mu,\omega}$. The proof of Proposition \ref{prop-mini-prob} is now complete.
\end{proof}

Let $\Ac^+_{\mu,\omega}$ be as in \eqref{Ac-omega}. Our next aim is to show the following coercivity property of $\Ac^+_{\mu,\omega}$. 

\begin{proposition}[\bf Coercivity] \label{prop-coer} ~\\
	Let $d\geq 1$, $\mu\geq 0$, $0<b<\min\{d,4\}$, $\alpha>\frac{8-2b}{d}$, $\alpha<\frac{8-2b}{d-4}$ if $d\geq 5$, and $\omega>0$. Then there exists $C>0$ such that for any $f\in \Ac^+_{\mu,\omega}$,
	\begin{align} \label{coer-prope}
	G_\mu(f) \geq \max \left\{ C\int |x|^{-b}|f(x)|^{\alpha+2} dx, 2\(m_{\mu,\omega} - S_{\mu,\omega}(f)\) \right\}.
	\end{align}
\end{proposition}

Before giving the proof of Proposition \ref{prop-coer}, we give the following observation.

\begin{observation} \label{obse-2}
	We have that $\Ac^+_{\mu,\omega}$ is an open set of $H^2(\R^d)$.
\end{observation}

\begin{proof}[Proof of Observation \ref{obse-2}]
	We observe that
	\[
	\Ac^+_{\mu,\omega} = \left\{ f\in H^2(\R^d) : S_{\mu,\omega}(f)<m_{\mu,\omega}, G_\mu(f) > 0\right\} \cup \{0\}.
	\]
	In fact, if $S_{\mu,\omega}(f) < m_{\mu,\omega}$ and $G_\mu(f)=0$, then by the definition of $m_{\mu,\omega}$, we must have $f\equiv 0$. 
	
	To see that $\Ac^+_{\mu,\omega}$ is an open set of $H^2(\R^d)$, it suffices to prove that $0$ is an interior point of $\Ac^+_{\mu,\omega}$. That is there exists $\delta>0$ small such that if $f \in H^2(\R^d)\backslash \{0\}$ satisfying $\|f\|_{H^2} <\delta$, then $f\in \Ac^+_{\mu,\omega}$. As $m_{\mu,\omega}>0$, we have
	\[
	S_{\mu,\omega}(f) \leq \frac{1}{2}\|\Delta f\|^2_{L^2} + \frac{\mu}{2} \|\nabla f\|^2_{L^2} + \frac{\omega}{2}\|f\|^2_{L^2} \leq C(\mu,\omega) \|f\|^2_{H^2} \leq C(\mu,\omega) \delta <m_{\mu,\omega}
	\]
	provided that $\delta>0$ is taken sufficiently small. On the other hand, by \eqref{GN-ineq}, we have
	\begin{align*}
	G_\mu(f) &= 2\|\Delta f\|^2_{L^2} + \mu \|\nabla f\|^2_{L^2} - \frac{d\alpha+2b}{2(\alpha+2)}\int |x|^{-b} |f(x)|^{\alpha+2} dx \\
	&\geq 2\|\Delta f\|^2_{L^2} + \mu \|\nabla f\|^2_{L^2} - C \|\Delta f\|^{\frac{d\alpha+2b}{4}}_{L^2} \|f\|^{\frac{8-2b-(d-4)\alpha}{4}}_{L^2} \\
	&\geq 2\|\Delta f\|^2_{L^2} + \mu \|\nabla f\|^2_{L^2} - C \|\Delta f\|^2_{L^2} \|f\|^\alpha_{H^2}\\
	&\geq \(2-C\delta^\alpha\) \|\Delta f\|^2_{L^2} + \mu \|\nabla f\|^2_{L^2} \geq 0
	\end{align*}
	for $\delta>0$ small. The proof is complete.
\end{proof}

\begin{proof}[Proof of Proposition \ref{prop-coer}]
	We only consider $f\ne 0$ since the case $f=0$ is trivial. We consider two cases.
	
	{\bf Case 1.} Assume that
	\[
	12 \|\Delta f\|^2_{L^2} + 4\mu \|\nabla f\|^2_{L^2} - \frac{(d\alpha+2b+4)(d\alpha+2b)}{4(\alpha+2)}  \int |x|^{-b} |f(x)|^{\alpha+2} dx \geq 0.
	\]
	It follows that
	\begin{align*}
	12 \|\Delta f\|^2_{L^2} + 4\mu \|\nabla f\|^2_{L^2} \geq  \frac{(d\alpha+2b+4)(d\alpha+2b)}{4(\alpha+2)}  \int |x|^{-b} |f(x)|^{\alpha+2} dx.
	\end{align*}
	Hence
	\begin{align*}
	G_\mu(f) &= 2\|\Delta f\|^2_{L^2} +\mu\|\nabla f\|^2_{L^2} -\frac{d\alpha+2b}{2(\alpha+2)} \int |x|^{-b} |f(x)|^{\alpha+2} dx \\
	&\geq \frac{2(d\alpha+2b-8)}{d\alpha+2b+4} \|\Delta f\|^2_{L^2} + \frac{(d\alpha+2b-4)\mu}{d\alpha+2b+4} \|\nabla f\|^2_{L^2} \\
	&\geq C\int |x|^{-b} |f(x)|^{\alpha+2} dx.
	\end{align*}
	
	{\bf Case 2.} Assume that
	\begin{align} \label{case-2}
	12 \|\Delta f\|^2_{L^2} + 4\mu \|\nabla f\|^2_{L^2} - \frac{(d\alpha+2b+4)(d\alpha+2b)}{4(\alpha+2)}  \int |x|^{-b} |f(x)|^{\alpha+2} dx < 0.
	\end{align}
	Denote 
	\[
	\phi(\lambda):= E_{\mu,\omega}(f_\lambda) = \frac{\lambda^4}{2}\|\Delta f\|^2_{L^2} + \frac{\mu\lambda^2}{2} \|\nabla f\|^2_{L^2} + \frac{\omega}{2} \|f\|^2_{L^2} - \frac{\lambda^{\frac{d\alpha+2b}{2}}}{\alpha+2} \int |x|^{-b} |f(x)|^{\alpha+2} dx.
	\]
	We have
	\[
	\phi'(\lambda) = 2\lambda^3\|\Delta f\|^2_{L^2} +\mu \lambda \|\nabla f\|^2_{L^2} -\frac{d\alpha+2b}{2(\alpha+2)} \lambda^{\frac{d\alpha+2b-2}{2}} \int |x|^{-b} |f(x)|^{\alpha+2} dx = \frac{G_\mu(f_\lambda)}{\lambda}
	\]
	and
	\begin{align*}
	(\lambda \phi'(\lambda))' &= 8 \lambda^3 \|\Delta f\|^2_{L^2} + 2\mu \lambda \|\nabla f\|^2_{L^2} - \frac{(d\alpha+2b)^2}{4(\alpha+2)} \lambda^{\frac{d\alpha+2b-2}{2}} \int |x|^{-b} |f(x)|^{\alpha+2} dx \\
	&= -2 \phi'(\lambda) + \lambda^3 \psi(\lambda),
	\end{align*}
	where
	\[
	\psi(\lambda):= 12 \|\Delta f\|^2_{L^2} + 4\mu \lambda^{-2} \|\nabla f\|^2_{L^2} - \frac{(d\alpha+2b+4)(d\alpha+2b)}{4(\alpha+2)} \lambda^{\frac{d\alpha+2b-8}{2}} \int |x|^{-b} |f(x)|^{\alpha+2} dx.
	\]
	Note that $\psi(1)<0$ due to \eqref{case-2}. As $\psi(\lambda)$ is strictly decreasing on $(0,+\infty)$, we have $\psi(\lambda) \leq \psi(1)<0$ for all $\lambda \geq 1$, hence
	\begin{align} \label{coer-prope-prof-2}
	(\lambda \phi'(\lambda))' \leq -2\phi'(\lambda), \quad \forall \lambda \geq 1. 
	\end{align}
	As $f\ne 0$, we have $G_\mu(f)>0$. By Observation \eqref{obse-1}, there exists $\lambda_0>1$ such that $G_\mu(f_{\lambda_0})=0$ or $\lambda_0 \phi'(\lambda_0)=0$. By the definition of $m_{\mu,\omega}$, we have $\phi(\lambda_0)=S_{\mu,\omega}(f_{\lambda_0}) \geq m_{\mu,\omega}$. Integrating \eqref{coer-prope-prof-2} over $(1,\lambda_0)$, we obtain
	\[
	G_\mu(f) \geq 2\(S_{\mu,\omega}(f_{\lambda_0})- S_{\mu,\omega}(f)\) \geq 2\(m_{\mu,\omega}-S_{\mu,\omega}(f)\).
	\]
	Collecting the above cases, we finish the proof of Proposition \ref{prop-coer}.
\end{proof}

\section{Global existence and energy scattering}
\label{S5}
\setcounter{equation}{0}

In this section, we give the proofs of the global existence and the energy scattering given in Theorem \ref{theo-gwp} and Theorem \ref{theo-scat}

\begin{proof}[Proof of Theorem \ref{theo-gwp}]
	The proof is divided into several steps.
	
	{\bf Step 1.} We first show the invariance of $\Ac^+_{\mu,\omega}$ under the flow of \eqref{IBNLS}. By the conservation of mass and energy, we have $S_{\mu,\omega}(u(t))= S_{\mu,\omega}(u_0) <m_{\mu,\omega}$ for all $t\in I_{\max}$, where $I_{\max}$ is the maximal time interval of existence. We now prove that $G_\mu(u(t))\geq 0$ for all $t\in I_{\max}$. Assume by contradiction that there exists $t_0\in I_{\max}$ such that $G_\mu(u(t_0)) <0$. By the continuity, there exists $t_1 \in I_{\max}$ such that $G_\mu(u(t_1)) =0$, which, by the definition of $m_{\mu,\omega}$, yields $S_{\mu,\omega}(u(t_1)) \geq m_{\mu,\omega}$ which is a contradiction. This shows that $u(t) \in \Ac^+_{\mu,\omega}$ for all $t\in I_{\max}$. We also have
	\begin{align*}
	\frac{d\alpha+2b-8}{2(d\alpha+2b)} \|\Delta u(t)\|^2_{L^2} &+ \frac{(d\alpha+2b-4)\mu}{2(d\alpha+2b)} \|\nabla u(t)\|^2_{L^2} + \frac{\omega}{2} \|u(t)\|^2_{L^2} \\
	&=S_{\mu,\omega}(u(t)) - \frac{2}{d\alpha+2b} G_\mu(u(t)) \\
	&\leq S_{\mu,\omega}(u(t)) <m_{\mu,\omega}, \quad \forall t\in I_{\max}
	\end{align*}
	which implies
	\begin{align} \label{boun-solu}
	\|u(t)\|_{H^2} \leq C, \quad \forall t\in I_{\max}.
	\end{align}
	The blow-up alternative ensures that the solution exists globally in time.
	
	{\bf Step 2.} We next show that there exists $R_0>0$ sufficiently large such that $\chi_R u(t) \in \Ac^+_{\mu,\omega}$ for all $R\geq R_0$ and all $t\in \R$, where $\chi_R(x)=\chi(x/R)$ with $\chi\in C^\infty_0(\R^d)$ satisfying $0\leq \chi \leq 1$, $\chi(x)=1$ on $|x|\leq 1/2$, and $\chi(x)=0$ on $|x|\geq 1$. In particular, Proposition \ref{prop-coer} implies that
	\begin{align} \label{coer-est-solu}
	G_\mu(\chi_Ru(t)) \geq C \int |x|^{-b} |\chi_R(x) u(t,x)|^{\alpha+2} dx, \quad \forall R\geq R_0, \quad \forall t\in \R.
	\end{align}
	To this end, we recall the following useful identities: (see \cite{DM} for \eqref{iden-1} and \cite{Dinh-NON} for \eqref{iden-2}) 
	\begin{align} \label{iden-1}
	\int |\nabla(\chi f)|^2 dx = \int \chi^2 |\nabla f|^2 dx -\int \chi\Delta \chi |f|^2dx
	\end{align}
	and
	\begin{align}
	\int |\Delta(\chi f)|^2 dx &= \int \chi^2 |\Delta f|^2 dx \nonumber \\
	&\quad + \int (\Delta \chi)^2 |f|^2 dx - 4\sum_{k,l} \rea \int \chi \partial_k f \partial^2_{kl} \chi \partial_l \overline{f} dx + 2 \int |\nabla \chi|^2 |\nabla f|^2 dx \nonumber\\
	&\quad + 2\int \chi \Delta \chi |\nabla f|^2 dx + 2 \rea \int \chi \Delta f \Delta \chi \overline{f} dx + 4 \rea \int \nabla \chi \cdot \nabla f \Delta \chi \overline{f} dx.\label{iden-2}
	\end{align}
	From this and \eqref{boun-solu}, we have
	\begin{align*}
	\|\Delta (\chi_Ru(t))\|^2_{L^2} &= \int \chi_R^2 |\Delta u(t)|^2 dx + O(R^{-2}), \\ 
	\|\nabla (\chi_R u(t))\|^2_{L^2} &= \int \chi^2_R |\nabla u(t)|^2 dx + O(R^{-2}), 
	\end{align*}
	for all $t\in \R$. In particular, we have
	\begin{align*}
	S_{\mu,\omega}(\chi_R u(t)) &= \int \chi_R^2 \(\frac{1}{2}|\Delta u(t)|^2 + \frac{\mu}{2}|\nabla u(t)|^2 +\frac{\omega}{2}|u(t)|^2\) dx \\
	&\quad -\frac{1}{\alpha+2} \int |x|^{-b}|\chi_R u(t)|^{\alpha+2} dx + O(R^{-2}+R^{-b}) \\
	&\leq S_{\mu,\omega} (u(t)) + O(R^{-2}+ R^{-b}) \\
	&=S_{\mu,\omega}(u_0) + O(R^{-2}+R^{-b}), \quad \forall t\in\R,
	\end{align*}
	where we have used that
	\begin{align*}
	\Big|\int |x|^{-b} |\chi_R u(t)|^{\alpha+2}dx - \int |x|^{-b} |u(t)|^{\alpha+2} dx\Big| &\leq C\int_{|x|\geq R/2} |x|^{-b} |u(t)|^{\alpha+2} dx \\
	&\leq CR^{-b} \|u(t)\|^{\alpha+2}_{L^{\alpha+2}} \\
	&\leq CR^{-b}\|u(t)\|^{\alpha+2}_{H^2} \\
	&\leq CR^{-b}, \quad \forall t\in \R. 
	\end{align*}
	Thus we get
	\begin{align} \label{est-S-omega}
	S_{\mu,\omega}(\chi_R u(t)) <m_{\mu,\omega} - \delta/2, \quad \forall R\geq R_0, \quad \forall t\in \R
	\end{align}
	provided that $R_0>0$ is taken large so that $O(R_0^{-2}+R_0^{-b}) <\delta/2$, where 
	\[
	\delta:=m_{\mu,\omega}-S_{\mu,\omega}(u_0)>0.
	\]
	By increasing $R_0$ if necessary, it suffices to prove that $G_\mu(\chi_Ru(t))\geq 0$ for all $R\geq R_0$ and all $t\in \R$. Suppose by contradiction that there exist $R_1 \geq R_0$ and $t_1 \in  \R$ such that $G_\mu (\chi_{R_1} u(t_1)) < 0$. Note that $\lim_{R\to \infty} G_\mu (\chi_R u(t_1)) = G_\mu (u(t_1)) > 0$ (here we can assume $u_0 \ne 0$, so is $u(t_1)$, otherwise the claim is trivial). Let $R_2 > R_1$ be the smallest value such that $G_\mu (\chi_{R_2} u(t_1)) = 0$, that is, $G_\mu (\chi_R u(t_1)) < 0$ for all $R \in [R_1 ,R_2)$. Since $G_\mu (\chi_{R_2} u(t_1)) = 0$ and $S_{\mu,\omega} (\chi_{R_2} u(t_1)) < m_{\mu,\omega}$ (see \eqref{est-S-omega}), we have $\chi_{R_2} u(t_1) \in \Ac^+_{\mu,\omega}$. Because $\Ac^+_{\mu,\omega}$ is open in $H^2(\R^d)$, we deduce, by the continuity, that $G_\mu (\chi_R u(t_1)) \geq 0$ for $R < R_2$ and $R$ close to $R_2$ . This contradicts to the choice of $R_2$. We therefore prove that $\chi_R u(t)\in \Ac^+_{\mu,\omega}$ for all $R\geq R_0$ and all $t\in \R$. 
	
	To see \eqref{coer-est-solu}, we apply Proposition \ref{prop-coer} to get
	\[
	G_\mu(\chi_R u(t)) \geq \max \left\{ C\int |x|^{-b}|\chi_R u(t)|^{\alpha+2} dx, 2\(m_{\mu,\omega} - S_{\mu,\omega}(\chi_R u(t))\) \right\}
	\]
	for all $R\geq R_0$ and all $t\in \R$. This gives \eqref{coer-est-solu} since
	\[
	2\(m_{\mu,\omega} - S_{\mu,\omega}(\chi_Ru(t))\) >\delta
	\]
	and
	\[
	\int |x|^{-b} |\chi_R u(t)|^{\alpha+2} dx \leq \|u(t)\|^{\alpha+2}_{H^2} \leq C
	\]
	for all $R>0$ and all $t\in \R$ due to \eqref{boun-solu}. 
	
	{\bf Step 3.} We prove the space-time estimates \eqref{space-time-est-non-rad} and \eqref{space-time-est-rad}. To do this, we introduce $\zeta: [0,\infty) \to [0,2]$ a smooth function satisfying
	\[
	\zeta(r)= \left\{
	\begin{array}{cl}
	2 &\text{if } 0\leq r\leq 1, \\
	0 &\text{if } r\geq 2,
	\end{array}
	\right.
	\]
	and define
	\[
	\vartheta(r):= \int_0^r \int_0^s \zeta(\tau) d\tau ds. 
	\]
	For $R>0$, we define a radial function
	\[
	\varphi_R(x)=\varphi_R(r):= R^2 \vartheta(r/R), \quad r=|x|.
	\]
	One readily check that
	\[
	0\leq \varphi''_R(r) \leq 2, \quad 0\leq \frac{\varphi'_R(r)}{r} \leq 2, \quad 2d -\Delta \varphi_R(x) \geq 0, \quad \forall r\geq 0, \quad \forall x \in \R^d. 
	\]
	We define the virial quantity
	\[
	\Mcal_{\varphi_R}(t):= 2\ima \int \nabla \varphi_R \cdot \nabla u(t,x) \overline{u}(t,x) dx. 
	\]
	By \eqref{boun-solu}, we have
	\begin{align} \label{est-M-varphi-R}
	|\Mcal_{\varphi_R}(t)| \leq \|\nabla \varphi_R\|_{L^\infty} \|u(t)\|_{L^2} \|\nabla u(t)\|_{L^2} \leq CR, \quad \forall t\in \R.
	\end{align}
	Moreover, we have the following virial identity due to \cite[Lemma 2.3]{Dinh-AML}:
	\begin{align} \label{viri-iden}
	\begin{aligned}
	\frac{d}{dt} \Mcal_{\varphi_R}(t) &= \int \Delta^3 \varphi_R |u(t)|^2 dx - 2\int \Delta^2\varphi_R |\nabla u(t)|^2 dx \\
	&\quad -4 \sum_{k,l} \int \partial^2_{kl} \Delta \varphi_R \partial_k \overline{u}(t) \partial_l u(t) dx - \mu \int \Delta^2\varphi_R |u(t)|^2 dx \\
	&\quad +8 \int \frac{\varphi'_R}{r} \sum_k |\nabla \partial_k u(t)|^2 + \(\frac{\varphi''_R}{r^2}-\frac{\varphi'_R}{r^3} \) \sum_k |x \cdot \nabla \partial_k u(t)|^2 dx \\
	&\quad +4 \mu \int \frac{\varphi'_R}{r}|\nabla u(t)|^2 + \( \frac{\varphi''_R}{r^2}-\frac{\varphi'_R}{r^3} \)|x\cdot \nabla u(t)|^2 dx \\
	&\quad -\frac{1}{\alpha+2} \int \(2\alpha \varphi''_R + (2(d-1)\alpha+4b)\frac{\varphi'_R}{r}\) |x|^{-b} |u(t)|^{\alpha+2} dx
	\end{aligned}
	\end{align}
	for all $t\in \R$. As $\varphi_R(x)=|x|^2$ for $|x|\leq R$, we have
	\begin{align*}
	\frac{d}{dt}\Mcal_{\varphi_R}(t) &= 4\left(2\int_{|x|\leq R} |\Delta u(t)|^2 +  \mu \int_{|x|\leq R} |\nabla u(t)|^2 -\frac{d\alpha+2b}{2(\alpha+2)} \int_{|x|\leq R} |x|^{-b} |u(t)|^{\alpha+2} dx\right) \\
	&\quad + \int \Delta^3\varphi_R |u(t)|^2 dx - 2\int \Delta^2\varphi_R |\nabla u(t)|^2 dx \\
	&\quad -4 \sum_{k,l} \int \partial^2_{kl}\Delta \varphi_R \partial_k \overline{u}(t) \partial_l u(t) dx -\mu \int \Delta^2\varphi_R |u(t)|^2 dx \\
	&\quad +8 \int_{|x|>R} \frac{\varphi'_R}{r} \sum_k |\nabla \partial_k u(t)|^2 + \(\frac{\varphi''_R}{r^2}-\frac{\varphi'_R}{r^3}\) \sum_k |x\cdot \nabla \partial_k u(t)|^2 dx \\
	&\quad +4 \mu \int_{|x|>R} \frac{\varphi'_R}{r}|\nabla u(t)|^2 + \( \frac{\varphi''_R}{r^2}-\frac{\varphi'_R}{r^3} \)|x\cdot \nabla u(t)|^2 dx \\
	&\quad - \frac{1}{\alpha+2} \int_{|x|>R} \(2\alpha \varphi''_R + (2(d-1)\alpha+4b)\frac{\varphi'_R}{r}\) |x|^{-b} |u(t)|^{\alpha+2} dx.
	\end{align*}
	By integration by parts, we note that
	\[
	\sum_k \int |\nabla \partial_k u(t)|^2 dx = \sum_{k,l} \int \partial^2_{kl} \overline{u}(t) \partial^2_{kl} u(t) dx = \sum_{k,l} \int \partial^2_k \overline{u}(t) \partial^2_l u(t) dx = \int |\Delta u(t)|^2 dx.
	\]
	By the choice of $\varphi_R$, we infer from \eqref{boun-solu} that
	\begin{align*}
	\Big| \int \Delta^3 \varphi_R |u(t)|^2 dx \Big| &\lesssim R^{-4}, \\
	\Big| \int \Delta^2 \varphi_R |\nabla u(t)|^2 dx\Big| + \Big|\sum_{k,l} \int \partial^2_{k,l} \Delta \varphi_R \partial_k \overline{u}(t) \partial_l u(t) dx\Big| &+ \Big|\int \Delta^2\varphi_R |u(t)|^2 dx \Big| \lesssim R^{-2},
	\end{align*}
	for all $t\in \R$. By Cauchy-Schwarz' inequality, we have
	\begin{align*}
	\frac{\varphi'_R}{r} |\nabla \partial_k u|^2 &+ \(\frac{\varphi''_R}{r^2}-\frac{\varphi'_R}{r^3}\) |x\cdot \nabla \partial_k u|^2 \\
	&=\frac{\varphi''_R}{r} |x\cdot \nabla_k u|^2 + \frac{\varphi'_R}{r}|\nabla \partial_k u|^2 - \frac{\varphi'_R}{r^3}|x\cdot \nabla \partial_k u|^2 \\
	&\geq \frac{\varphi''_R}{r} |x\cdot \nabla_k u|^2 \geq 0
	\end{align*}
	as $\varphi'_R, \varphi''_R\geq 0$. Similarly, we have
	\[
	\frac{\varphi'_R}{r} |\nabla u|^2 + \(\frac{\varphi''_R}{r^2}-\frac{\varphi'_R}{r^3}\) |x\cdot \nabla u|^2 \geq \frac{\varphi''_R}{r^2} |x \cdot \nabla u|^2 \geq 0.
	\]
	Moreover, since $\varphi''_R$ and $\frac{\varphi'_R}{r}$ are bounded on $\R^d$, we infer from Sobolev embedding and \eqref{boun-solu} that
	\[
	\Big|\int_{|x|>R} \(2\alpha \varphi''_R + (2(d-1)\alpha+4b)\frac{\varphi'_R}{r}\) |x|^{-b} |u(t)|^{\alpha+2} dx \Big| \lesssim R^{-b}.
	\]
	Thus we arrive at
	\begin{align*}
	\frac{d}{dt}\Mcal_{\varphi_R}(t) &\geq 4\left(2\int_{|x|\leq R} |\Delta u(t)|^2 +  \mu \int_{|x|\leq R} |\nabla u(t)|^2 -\frac{d\alpha+2b}{2(\alpha+2)} \int_{|x|\leq R} |x|^{-b} |u(t)|^{\alpha+2} dx\right) \\
	&\quad + O(R^{-2} + R^{-b}) 
	\end{align*}
	for all $t\in \R$. On the other hand, we have from \eqref{iden-1} and \eqref{iden-2} that
	\begin{align*}
	\int |\nabla (\chi_R u)|^2 dx &=\int \chi^2_R |\nabla u|^2 dx + O(R^{-2}) \\
	&= \int_{|x|\leq R} |\nabla u|^2 dx - \int_{R/2\leq |x|\leq R} (1-\chi^2_R) |\nabla u|^2 dx + O(R^{-2}), \\
	\int |\Delta(\chi_R u)|^2 dx &= \int \chi^2_R |\Delta u|^2 dx +  O\left(R^{-2}\right) \\
	&= \int_{|x| \leq R} |\Delta u|^2 dx - \int_{R/2 \leq |x| \leq R} (1-\chi^2_R) |\Delta u|^2 dx + O(R^{-2}).
	\end{align*}
	We also have
	\begin{align*}
	\int |x|^{-b} |\chi_R u|^{\alpha+2} dx &= \int_{|x| \leq R} |x|^{-b} |u|^{\alpha+2} - \int_{R/2 \leq |x| \leq R} (1-\chi^{\alpha+2}_R) |x|^{-b} |u|^{\alpha+2}dx \\
	&= \int_{|x| \leq R} |x|^{-b} |u|^{\alpha+2} + O(R^{-b}).
	\end{align*}
	This implies that
	\begin{align*}
	2\int_{|x| \leq R} |\Delta u|^2 dx &+\mu \int_{|x|\leq R} |\nabla u(t)|^2 dx - \frac{d\alpha+2b}{2(\alpha+2)} \int_{|x| \leq R} |x|^{-b} |u|^{\alpha+2} dx \\
	&= 2\int |\Delta(\chi_R u)|^2 dx  +\mu \int |\nabla(\chi_R u)|^2 dx -\frac{d\alpha+2b}{2(\alpha+2)} \int |x|^{-b} |\chi_R u|^{\alpha+2} dx \\
	&\quad + 2\int (1-\chi^2_R) |\Delta u|^2 dx +\mu\int (1-\chi^2_R) |\nabla u|^2 dx + O(R^{-2} +R^{-b}) \\
	&\geq G_\mu(\chi_R u) + O(R^{-2}+R^{-b})
	\end{align*}
	as $0\leq \chi_R \leq 1$ and $\mu\geq 0$. We obtain
	\[
	\frac{d}{dt} \Mcal_{\varphi_R}(t) \geq 4 G_\mu(\chi_Ru(t)) + O(R^{-2}+R^{-b}), \quad \forall t\in \R. 
	\]
	Thanks to \eqref{coer-est-solu}, we have for all $R\geq R_0$,
	\[
	C \int |x|^{-b} |\chi_R u(t)|^{\alpha+2} dx \leq \frac{d}{dt} \Mcal_{\varphi_R}(t) + O(R^{-2}+R^{-b}), \quad \forall t\in \R. 
	\]
	Thus for any interval $I \in \R$, we infer from \eqref{est-M-varphi-R} that
	\[
	\int_I \int_{\R^d} |x|^{-b} |\chi_R u(t)|^{\alpha+2} dxdt \lesssim R + (R^{-2}+R^{-b})|I|. 
	\]
	By the choice of $\chi_R$, we deduce
	\[
	\int_I \int_{|x|\leq R/2} |x|^{-b} |u(t)|^{\alpha+2} dxdt \lesssim R + (R^{-2}+R^{-b}) |I|.
	\]
	On the other hand, we have
	\[
	\int_{|x|\geq R/2} |x|^{-b} |u(t)|^{\alpha+2} dx \lesssim R^{-b} \|u(t)\|^{\alpha+2}_{H^2} \lesssim R^{-b}
	\]
	which yields
	\[
	\int_I \int_{\R^d} |x|^{-b} |u(t)|^{\alpha+2} dxdt \lesssim R + (R^{-2}+R^{-b}) |I| \lesssim R + R^{-\min\{2,b\}}|I|.
	\]
	If $|I|$ is sufficiently large depending on $R_0$, we take $R=|I|^{\frac{1}{\min\{2,b\}}}$ and get \eqref{space-time-est-non-rad}. Otherwise, if $|I|$ is sufficiently small, then 
	\[
	\int_I \int_{\R^d} |x|^{-b} |u(t)|^{\alpha+2} dx dt \lesssim |I| \lesssim |I|^{\frac{1}{1+\min\{2,b\}}}.
	\]
	Finally, if we assume in addition that the solution is radially symmetric, then we can improve the space-time estimate \eqref{space-time-est-non-rad}. More precisely, by using the radial Sobolev embedding (see e.g., \cite{Strauss}): for $d\geq 2$,
	\[
	\||x|^{\frac{d-1}{2}} f\|_{L^\infty} \lesssim \|f\|_{H^1}, \quad \forall f\in H^1_{\rad}(\R^d)
	\]
	and \eqref{boun-solu}, we have
	\begin{align*}
	\left|\int_{|x|\geq R/2} |x|^{-b} |u(t,x)|^{\alpha+2} dx\right|&\lesssim \(\sup_{|x|>R/2}|x|^{-b} |u(t,x)|^\alpha\) \|u(t)\|^2_{L^2} \\
	&\lesssim R^{-\frac{(d-1)\alpha +2b}{2}} \|u(t)\|^\alpha_{H^1} \|u(t)\|^2_{L^2} \\
	&\lesssim R^{-\frac{(d-1)\alpha+2b}{2}} \\
	&\lesssim R^{-2}.
	\end{align*}
	Note that $\frac{(d-1)\alpha+2b}{2}>2$ due to $\alpha>\frac{8-2b}{d}$ and $d\geq 2$. In particular, we have
	\[
	\frac{d}{dt}\Mcal_{\varphi_R}(t) \geq 4G_\mu(\chi_R u(t)) + O(R^{-2}), \quad \forall t\in \R.
	\]
	Arguing as above, we conclude \eqref{space-time-est-rad}. The proof is complete.
\end{proof}

\begin{proof}[Proof of Theorem \ref{theo-scat}]
	Let $a\in \R$ and $(k,r)$ be as in \eqref{qrkm}. By Sobolev embedding and Strichartz estimates in Lorentz spaces, and \eqref{boun-solu}, there exists a constant $E$ such that
	\[
	\|U_\mu(t-a) u(a)\|_{L^k_t(\R, L^{r,2}_x)} \leq C\|\scal{\Delta} U_\mu(t-a) u(a)\|_{L^k_t(\R, L^{\overline{r},2}_x)} \leq C\|u(a)\|_{H^2_x} \leq E,
	\]
	where 
	\[
	\overline{r} = \frac{2d\alpha(\alpha+2)}{d\alpha^2 +4(d-2)\alpha-16+4b}
	\]
	with $(k,\overline{r})\in B$. 
	
	Given a small constant $\vareps>0$. There exist an integer number $J=J(\vareps, E)$ and disjoint intervals $I_j, j=1,..., J$ such that 
	\[
	\|U_\mu(t-a) u(a)\|_{L^k_t(I_j, L^{r,2}_x)} <\vareps, \quad \forall j=1, ..., J.
	\]
	In the following, the constant $C$ may vary but it depends only on $E$. We will show that there exists $T=T(\vareps)>0$ large such that
	\begin{align} \label{est-I-j}
	\|u\|^k_{L^k_t(I_j, L^{r,2}_x)} \leq CT, \quad \forall j=1,...,J.
	\end{align}
	Summing over all sub-intervals $I_j$, we get the boundedness of $\|u\|_{L^k_t(\R, L^{r,2}_x)}$ which, by Lemma \ref{lem-scat-cond}, yields the energy scattering.
	
	By the $H^2$-boundedness (see \eqref{boun-solu}), Sobolev inequality in space, and H\"older inequality in time, we have
	\begin{align} \label{est-Sob-Hol}
	\|u\|^k_{L^k_t(I,L^{r,2}_x)} \leq C|I|.
	\end{align}
	From this, it suffices to show \eqref{est-I-j} for sub-intervals $I_j$ with $|I_j|>2T$. Fix such an interval $I=(a,c)$ with $|I|>2T$. We will show that
	\begin{align} \label{est-I}
	\|u\|^k_{L^k_t(I,L^{r,2}_x)} \leq CT.
	\end{align}
	The proof of \eqref{est-I} is reduced to show that there exists $t_1 \in (a,a+T)$ so that 
	\begin{align} \label{est-T}
	\left\|\int_a^{t_1} U_\mu(t-\tau) (|x|^{-b}|u|^\alpha u)(\tau) d\tau\right\|_{L^k_t([t_1, +\infty), L^{r,2}_x)} \leq C\vareps^\gamma
	\end{align}
	for some constant $\gamma>0$. Assume \eqref{est-T} for the moment, let us prove \eqref{est-I}. We write
	\[
	U_\mu(t-t_1) u(t_1) = U_\mu(t-a) u(a) +i \int_a^{t_1} U_\mu(t-\tau)(|x|^{-b}|u|^\alpha u)(\tau) d\tau
	\]
	and have
	\begin{align*}
	\|U_\mu(t-t_1) u(t_1)\|_{L^k_t([t_1,c], L^{r,2}_x)} &\leq \|U_\mu(t-a) u(a)\|_{L^k_t([t_1,c],L^{r,2}_x)}  \\
	&\quad + \left\| \int_a^{t_1} U_\mu(t-\tau) (|x|^{-b}|u|^\alpha u)(\tau) d\tau\right\|_{L^k_t([t_1,c], L^{r,2}_x)} \\
	&\leq \|U_\mu(t-a) u(a)\|_{L^k_t(I,L^{r,2}_x)}  \\
	&\quad + \left\| \int_a^{t_1} U_\mu(t-\tau) (|x|^{-b}|u|^\alpha u)(\tau) d\tau\right\|_{L^k_t([t_1,+\infty), L^{r,2}_x)} \\
	&\leq \vareps + C\vareps^\gamma.
	\end{align*}
	On the other hand, using
	\[
	u(t) = U_\mu(t-t_1) u(t_1) + i \int_{t_1}^t U_\mu(t-\tau) (|x|^{-b}|u|^\alpha u)(\tau) d\tau
	\]
	and Strichartz estimate \eqref{str-est-non-adm}, we have
	\begin{align*}
	\|u\|_{L^k_t([t_1,c], L^{r,2}_x)} &\leq \|U_\mu(t-t_1)u(t_1)\|_{L^k_t([t_1,c],L^{r,2}_x)} \\
	&\quad + \left\| \int_{t_1}^t U_\mu(t-\tau) (|x|^{-b}|u|^\alpha u)(\tau) d\tau \right\|_{L^k_t([t_1, c], L^{r,2}_x)} \\
	&\leq \vareps + C\vareps^\gamma + C\||x|^{-b}|u|^\alpha u\|_{L^{m'}_t([t_1,c], L^{r',2}_x)} \\
	&\leq \vareps  + C\vareps^\gamma + C\|u\|^{\alpha+1}_{L^k_t([t_1,c], L^{r,2}_x)}.
	\end{align*}
	Taking $\vareps>0$ small, the continuity argument yields
	\[
	\|u\|_{L^k_t([t_1,c], L^{r,2}_x)} \leq C\vareps + C\vareps^\gamma.
	\]
	Thus we get
	\[
	\|u\|^k_{L^k_t(I,L^{r,2}_x)} \leq \|u\|^k_{L^k_t([a,t_1],L^{r,2}_x)} + \|u\|^k_{L^k_t([t_1,c],L^{r,2}_x)} \leq \left(C\vareps + C\vareps^\gamma\right)^k + CT
	\]
	which proves \eqref{est-I}, where we have used \eqref{est-Sob-Hol} with $|t_1-a|<T$.
	
	We now prove \eqref{est-T}. To do this, we recall the following space-time estimate (see \eqref{space-time-est-non-rad} and \eqref{space-time-est-rad})
	\begin{align} \label{space-time-est}
	\int_I \int_{\R^d} |x|^{-b}|u(t,x)|^{\alpha+2} dx dt \leq C |I|^{\varrho},
	\end{align}
	where
	\[
	\varrho:= \left\{
	\begin{array}{cl}
	\frac{1}{1+\min\{2,b\}} &\text{for general solutions}, \\
	\frac{1}{3} &\text{for radial solutions}.
	\end{array}
	\right.
	\] 
	From \eqref{space-time-est}, we deduce that there exists $t_0 \in \left[a+\frac{T}{4},a+\frac{T}{2}\right]$ such that
	\begin{align} \label{est-t0}
	\int_{t_0}^{t_0+\vareps T^{1-\varrho}} \int_{\R^d}|x|^{-b}|u(t,x)|^{\alpha+2} dx dt \leq C\vareps.
	\end{align}
	In fact, we split $\left[a+\frac{T}{4},a+\frac{T}{2}\right]$ into $L=\vareps^{-1}T^\varrho$ sub-intervals $J_l$ of length $\vareps T^{1-\varrho}$. Thanks to \eqref{space-time-est}, we have
	\begin{align*}
	L \min_{1\leq l\leq L} \int_{J_l} \int_{\R^d} |x|^{-b}|u(t,x)|^{\alpha+2} dx dt &\leq \sum_{l=1}^L \int_{J_l} \int_{\R^d} |x|^{-b}|u(t,x)|^{\alpha+2} dx dt \\
	&=\int_{a+T/4}^{a+T/2} \int_{\R^d} |x|^{-b}|u(t,x)|^{\alpha+2} dx dt \\
	&\leq CT^\varrho.
	\end{align*}
	Thus there exists $l_0$ such that 
	\[
	\int_{J_{l_0}} \int_{\R^d} |x|^{-b}|u(t,x)|^{\alpha+2} dx dt \leq CT^\varrho L^{-1} =C\vareps.
	\]
	This shows the observation by setting $t_0=\inf J_{l_0}$.
	
	Now we take $t_1= t_0+\vareps T^{1-\varrho}$. As $t_0 <a+T/2$, we have $t_1<a+T$. To show \eqref{est-T}, we write
	\begin{align*}
	\int_a^{t_1} U_\mu(t-\tau) (|x|^{-b}|u|^\alpha u)(\tau) d\tau &= \int_a^{t_0} U_\mu(t-\tau) (|x|^{-b}|u|^\alpha u)(\tau) d\tau \\
	&\quad + \int_{t_0}^{t_1} U_\mu(t-\tau) (|x|^{-b}|u|^\alpha u)(\tau) d\tau =: F_1(t) + F_2(t).
	\end{align*}
	We refer $F_1$ and $F_2$ to as the distant past and recent past terms respectively.
	
	\vspace{2mm}
	
	\noindent \fbox{\bf The distant past} 
	
	\vspace{2mm}
	
	To estimate $F_1$, we consider several cases.
	
	\vspace{2mm}
	
	$\boxed{d\geq 5}$
	
	\vspace{2mm}
	
	For $\theta>0$ small to be fixed later, we pick $(e,n)$ satisfying
	\[
	\frac{1}{k} =\frac{2-\gamc}{2e} + \frac{\theta \gamc}{2}, \quad \frac{1}{r} =\frac{2-\gamc}{2n} + \frac{d-4-8\theta}{2d} \gamc.
	\]
	As $\frac{4}{k}+\frac{d}{r} =\frac{4-b}{\alpha}$ and $\gamc=\frac{d\alpha-8+2b}{2\alpha}$, we infer that
	\[
	\frac{4}{e}+\frac{d}{n} = \frac{d}{2}
	\]
	and 
	\[
	e=\frac{(2-\gamc)k}{2-k\gamc \theta} \to \alpha+2 \text{ as } \theta \to 0.
	\]
	In particular, $(e,n)\in B$ provided that $\theta>0$ is taken sufficiently small. We estimate
	\[
	\|F_1\|_{L^k_t([t_1,+\infty), L^{r,2}_x)} \leq \|F_1\|^{\frac{2-\gamc}{2}}_{L^e_t([t_1,+\infty), L^{n,2}_x)} \|F_1\|^{\frac{\gamc}{2}}_{L^{\frac{1}{\theta}}_t([t_1,+\infty),L^{\frac{2d}{d-4-8\theta},2}_x)}.
	\]
	By Strichartz estimates with $(e,n)\in B$ and the fact that
	\[
	F_1(t) = U_\mu(t-t_0)u(t_0) - U_\mu(t-a)u(a).
	\]
	we have 
	\[
	\|F_1\|_{L^e_t([t_1,+\infty),L^{n,2}_x)} \leq C.
	\] 
	To estimate $\|F\|_{L^{\frac{1}{\theta}}_t([t_1,+\infty), L^{\frac{2d}{d-4-8\theta},2}_x)}$, we first use dispersive estimates in Lorentz spaces \eqref{dis-est-lore} to have for $t>t_1$,
	\begin{align*}
	\|F_1(t)\|_{L^{\frac{2d}{d-4-8\theta},2}_x} &\leq C\int_a^{t_0} (t-\tau)^{-1-2\theta} \|(|x|^{-b}|u|^\alpha u)(\tau)\|_{L^{\frac{2d}{d+4+8\theta},2}_x} d\tau \\
	&\leq C\int_a^{t_0} (t-\tau)^{-1-2\theta} \|u(\tau)\|^{\alpha+1}_{H^2_x} d\tau \\
	&\leq C\int_a^{t_0} (t-\tau)^{-1-2\theta} d\tau \\
	&= \frac{C}{2\theta} \left((t-t_0)^{-2\theta} - (t-a)^{-2\theta}\right) \\
	&\leq \frac{C}{2\theta} (t-t_0)^{-2\theta},
	\end{align*}
	where we have used \eqref{non-est-4} to get the second line. It follows that
	\begin{align*}
	\|F_1\|_{L^{\frac{1}{\theta}}_t([t_1,+\infty), L^{\frac{2d}{d-4-8\theta},2}_x)} &\leq \frac{C}{2\theta} \left(\int_{t_1}^{+\infty} (t-t_0)^{-2}dt\right)^\theta \\
	&\leq \frac{C}{2\theta} (t_1-t_0)^{-\theta} \\
	&\leq \frac{C}{2\theta}(\vareps T^{1-\varrho})^{-\theta}.
	\end{align*}
	In particular, we get
	\begin{align}\label{est-T-1}
	\|F_1\|_{L^k_t([t_1,+\infty),L^{r,2}_x)} \leq \left(\frac{C}{2\theta}(\vareps T^{1-\varrho})^{-\theta}\right)^{\frac{\gamc}{2}}.
	\end{align}
	
	$\boxed{d=3,4}$
	
	\vspace{2mm}
	
	We estimate
	\[
	\|F_1\|_{L^k_t([t_1,+\infty), L^{r,2}_x)} \leq \|F\|^{1-\nu}_{L^e_t([t_1,+\infty),L^{n,2}_x)} \|F\|^\nu_{L^{\frac{1}{\theta}}_t([t_1,+\infty),L^{\frac{1}{\theta},2}_x)},
	\]
	where
	\[
	\frac{1}{k}=\frac{1-\nu}{e} + \nu \theta, \quad \frac{1}{r} =\frac{1-\nu}{n} +\nu \theta, \quad \nu:=\frac{d\alpha-8+2b}{d\alpha-2(d+4)\alpha \theta}.
	\]
	Since $d=3,4$ and $\alpha>\frac{8-2b}{d}$, it is readily seen that
	\[
	\frac{4}{e}+\frac{d}{n}=\frac{d}{2}
	\]
	and $\nu \in (0,1)$ for $\theta>0$ small. Moreover, we have
	\[
	e=\frac{(1-\nu)k}{1-k\nu \theta} \to \frac{16(4-b)\alpha(\alpha+2)}{d\alpha(8-2b-(d-4)\alpha)} >2 \text{ as } \theta \to 0
	\]
	which implies $(e,n) \in B$ for $\theta>0$ small. In particular, Strichartz estimates give
	\[
	\|F_1\|_{L^e_t([t_1,+\infty), L^{n,2}_x)} \leq C.
	\] 
	To estimate $\|F_1\|_{L^{\frac{1}{\theta}}_t([t_1,+\infty),L^{\frac{1}{\theta},2}_x)}$, we consider separately two cases: $d=4$ and $d=3$.
	
	$\bullet$ For $d=4$, we use dispersive estimates in Lorentz spaces \eqref{dis-est-lore} and \eqref{non-est-5} to have for $t>t_1$,
	\begin{align*}
	\|F_1(t)\|_{L^{\frac{1}{\theta},2}_x} &\leq C\int_a^{t_0} (t-\tau)^{-1+2\theta} \|(|x|^{-b}|u|^\alpha u)(\tau)\|_{L^{\frac{1}{1-\theta},2}_x} d\tau \\
	&\leq C\int_a^{t_0} (t-\tau)^{-1+2\theta} \|u(\tau)\|^{\frac{\alpha}{2}}_{H^2_x} \left(\int_{\R^d} |x|^{-b}|u(\tau,x)|^{\alpha+2}dx \right)^{1/2} d\tau \\
	&\leq C \int_a^{t_0} (t-\tau)^{-1+2\theta} \left(\int_{\R^d} |x|^{-b}|u(\tau,x)|^{\alpha+2}dx \right)^{1/2} d\tau\\
	&\leq C \left(\int_a^{t_0}(t-\tau)^{-2+4\theta} d\tau\right)^{1/2} \left(\int_a^{t_0}\int_{\R^d} |x|^{-b}|u(\tau,x)|^{\alpha+2} dxd\tau\right)^{1/2} \\
	&\leq C (t-t_0)^{-\frac{1-4\theta}{2}} (t_0-a)^{\frac{\varrho}{2}} \\
	&\leq C T^{\frac{\varrho}{2}} (t-t_0)^{-\frac{1-4\theta}{2}},
	\end{align*}
	where we have used \eqref{space-time-est} and $|t_0-a| \sim T$. It follows that
	\begin{align*}
	\|F_1\|_{L^{\frac{1}{\theta}}_t([t_1,+\infty),L^{\frac{1}{\theta},2}_x)} &\leq CT^{\frac{\varrho}{2}} \left(\int_{t_1}^{+\infty}(t-t_0)^{-\frac{1-4\theta}{2\theta}} dt\right)^\theta \\
	&\leq C T^{\frac{\varrho}{2}} (t_1-t_0)^{-\frac{1-6\theta}{2}} \\
	&\leq C T^{\frac{\varrho}{2}} (\vareps T^{1-\varrho})^{-\frac{1-6\theta}{2}}\\
	&\leq C \left(\vareps T^{1-\varrho-\frac{\varrho}{1-6\theta}}\right)^{-\frac{1-6\theta}{2}}.
	\end{align*}
	In particular, we have
	\begin{align}\label{est-T-2}
	\|F_1\|_{L^k_t([t_1,+\infty),L^{r,2}_x)} \leq C \left(\vareps T^{1-\varrho-\frac{\varrho}{1-6\theta}}\right)^{-\frac{(1-6\theta)\nu}{2}}.
	\end{align}
	
	$\bullet$ For $d=3$, we use \eqref{non-est-6} to have for $t>t_1$,
	\begin{align*}
	\|F_1(t)\|_{L^{\frac{1}{\theta},2}_x} &\leq C\int_a^{t_0} (t-\tau)^{-\frac{3}{4}(1-2\theta)} \|(|x|^{-b}|u|^\alpha u)(\tau)\|_{L^{\frac{1}{1-\theta},2}_x} d\tau \\
	&\leq C\int_a^{t_0} (t-\tau)^{-\frac{3}{4}(1-2\theta)} \|u(\tau)\|^{\frac{3\alpha-2}{8}}_{H^2_x} \left(\int_{\R^d} |x|^{-b}|u(\tau,x)|^{\alpha+2}dx \right)^{5/8} d\tau \\
	&\leq C \int_a^{t_0} (t-\tau)^{-\frac{3}{4}(1-2\theta)} \left(\int_{\R^d} |x|^{-b}|u(\tau,x)|^{\alpha+2}dx \right)^{5/8} d\tau\\
	&\leq C \left(\int_a^{t_0}(t-\tau)^{-2+4\theta} d\tau\right)^{3/8} \left(\int_a^{t_0}\int_{\R^d} |x|^{-b}|u(\tau,x)|^{\alpha+2} dxd\tau\right)^{5/8} \\
	&\leq C (t-t_0)^{-\frac{3(1-4\theta)}{8}} (t_0-a)^{\frac{5\varrho}{8}} \\
	&\leq C T^{\frac{5\varrho}{8}} (t-t_0)^{-\frac{3(1-4\theta)}{8}}.
	\end{align*}
	Thus we get
	\begin{align*}
	\|F_1\|_{L^{\frac{1}{\theta}}_t([t_1,+\infty),L^{\frac{1}{\theta},2}_x)} &\leq CT^{\frac{5\varrho}{8}} \left(\int_{t_1}^{+\infty}(t-t_0)^{-\frac{3(1-4\theta)}{8\theta}} dt\right)^\theta \\
	&\leq C T^{\frac{5\varrho}{8}} (t_1-t_0)^{-\frac{3-20\theta}{8}} \\
	&\leq C T^{\frac{5\varrho}{8}} (\vareps T^{1-\varrho})^{-\frac{3-20\theta}{8}}\\
	&\leq C \left(\vareps T^{1-\varrho-\frac{5\varrho}{3-20\theta}}\right)^{-\frac{3-20\theta}{8}}.
	\end{align*}
	In particular, we have
	\begin{align}\label{est-T-3}
	\|F_1\|_{L^k_t([t_1,+\infty),L^{r,2}_x)} \leq C\left(\vareps T^{1-\varrho-\frac{5\varrho}{3-20\theta}}\right)^{-\frac{(3-20\theta)\nu}{8}}.
	\end{align}
	
	\vspace{2mm}
	
	\noindent \fbox{\bf The recent past} 
	
	\vspace{2mm}
	
	By the same argument as in the proof of Lemma \ref{lem-str-est-non-adm}, we have
	\begin{align*}
	\|F_2\|_{L^k_t([t_1,+\infty),L^{r,2}_x)} &\leq C\||x|^{-b}|u|^\alpha u\|_{L^{m'}_t([t_0,t_1], L^{r',2}_x)} \\
	&\leq C \||x|^{-b}|u|^\alpha u\|_{L^{m'}_t([t_0,t_1], L^{r'}_x)}
	\end{align*}
	as $r'<2$. For $\vartheta, \delta>1$ to be chosen later, we estimate
	\begin{align*}
	\||x|^{-b}|u|^\alpha u\|_{L^{r'}_x(A)} \leq \||x|^{-\frac{b}{m'}} |u|^{\frac{\alpha+2}{m'}}\|_{L^{m'}_x} \||x|^{-b\left(1-\frac{1}{m'}\right)}\|_{L^\vartheta_x(A)} \||u|^{\alpha+1-\frac{\alpha+2}{m'}}\|_{L^\delta_x},
	\end{align*}
	where $A$ is either the unit ball $B(0,1)$ or its complement $B^c(0,1)$. We pick 
	\begin{align}\label{choi-vartheta}
	\frac{1}{\vartheta} = \frac{b\left(1-\frac{1}{m'}\right)}{d} \pm \theta
	\end{align}
	with the plus sign for $A=B(0,1)$ and the minus one for $A=B^c(0,1)$. This ensures that 
	\[
	\||x|^{-b\left(1-\frac{1}{m'}\right)}\|_{L^\vartheta_x(A)} <\infty.
	\] 
	Thus we have
	\[
	\||x|^{-b}|u|^\alpha u\|_{L^{r'}_x(A)} \leq \||x|^{-b} |u|^{\alpha+2}\|^{\frac{1}{m'}}_{L^1_x} \|u\|^{\alpha+1-\frac{\alpha+2}{m'}}_{L^{\delta\left(\alpha+1-\frac{\alpha+2}{m'}\right)}_x}
	\]
	hence
	\[
	\||x|^{-b}|u|^\alpha u\|_{L^{m'}_t([t_0,t_1],L^{r'}_x)} \leq C \left(\int_{t_0}^{t_1}\int_{\R^d} |x|^{-b} |u(t,x)|^{\alpha+2} dxt\right)^{\frac{1}{m'}} \|u\|^{\alpha+1-\frac{\alpha+2}{m'}}_{L^\infty_t([t_0,t_1],L^{\delta\left(\alpha+1-\frac{\alpha+2}{m'}\right)}_x)}.
	\]
	With the choice of $\vartheta$ in \eqref{choi-vartheta}, we have
	\[
	\frac{1}{\delta} = \frac{1}{r'}-\frac{1}{m'} - \frac{1}{\vartheta} = \frac{1}{m}-\frac{1}{r} - \frac{b}{dm} \mp \theta
	\]
	hence
	\[
	\delta\left(\alpha+1-\frac{\alpha+2}{m'}\right) = \frac{\frac{\alpha+2}{m}-1}{\frac{1}{m}-\frac{1}{r} - \frac{b}{dm} \mp \theta} \to \frac{d(\alpha+2)}{d-b}.
	\]
	As $\alpha>\frac{8-2b}{d}$ and $\alpha<\frac{8-2b}{d-4}$ if $d\geq 5$, the above limit satisfies $\frac{d(\alpha+2)}{d-b}>2$ and $\frac{d(\alpha+2)}{d-b}<\frac{2d}{d-4}$ if $d\geq 5$. By choosing $\theta>0$ sufficiently small, we can use Sobolev embedding and the $H^2$-boundedness \eqref{boun-solu} to get
	\[
	\|u\|_{L^\infty_t([t_0,t_1],L^{\delta\left(\alpha+1-\frac{\alpha+2}{m'}\right)}_x)} \leq C.
	\]
	Using \eqref{est-t0}, we obtain
	\begin{align}
	\|F_2\|_{L^{k}_t([t_1,+\infty),L^{r,2}_x)} \leq C \left(\int_{t_0}^{t_1}\int_{\R^d} |x|^{-b} |u(t,x)|^{\alpha+2} dxt\right)^{\frac{1}{m'}} \leq C \vareps^{\frac{1}{m'}}. \label{est-T-4}
	\end{align}
	
	Collecting \eqref{est-T-1}, \eqref{est-T-2}, \eqref{est-T-3}, and \eqref{est-T-4}, we prove \eqref{est-T} with $\gamma=\frac{1}{m'}$ provided that $T=T(\vareps)>0$ is taken sufficiently large so that
	\begin{align}\label{choi-T}
	\left\{
	\begin{array}{ll}
	(\vareps T^{1-\varrho})^{-\frac{\theta \gamc}{2}} <\vareps^\gamma &\text{if } d \geq 5, \\
	\left(\vareps T^{1-\varrho-\frac{\varrho}{1-6\theta}}\right)^{-\frac{(1-6\theta)\nu}{2}} <\vareps^\gamma &\text{if } d=4, \\
	\left(\vareps T^{1-\varrho-\frac{5\varrho}{3-20\theta}}\right)^{-\frac{(3-20\theta)\nu}{8}} <\vareps^\gamma &\text{if } d=3.
	\end{array}
	\right.
	\end{align}
	For general solutions, we have $\varrho=\frac{1}{1+\min\{2,b\}}$ with $0<b<\frac{d}{2}$.  When $d\geq 5$, as $\varrho<1$, we can take $\theta>0$ small and then choose $T=T(\vareps)$ sufficiently large so that \eqref{choi-T} holds. When $d=4$ and $1<b<2$, we have $\varrho=\frac{1}{1+b}<\frac{1}{2}$. Thus we can take $\theta>0$ small enough so that $1-\varrho-\frac{\varrho}{1-6\theta}>0$. Then \eqref{choi-T} is fulfilled if $T = T(\vareps)$ is taken large enough.
	
	For radial solutions, we have $\varrho=\frac{1}{3}$. When $d=4$ and $0<b\leq 1$, we have $1-\varrho-\frac{\varrho}{1-6\theta} >0$ provided that $\theta>0$ is small. When $d=3$ and $0<b<\frac{3}{2}$, we have $1-\varrho-\frac{5\varrho}{3-20\theta}>0$ by taking $\theta>0$ small. We are then able to choose $T=T(\vareps)$ large so that \eqref{choi-T} holds.  
	
	The proof of Theorem \ref{theo-scat} is now complete.
\end{proof}

\begin{proof}[Proof of Proposition \ref{prop-scat-defo}]
	The proof is similar (even simpler) to the focusing nonlinearity. First, due to the defocusing nature, $H^2$-solutions exist globally in time and satisfy
	\[
	\sup_{t\in \R} \|u(t)\|_{H^2_x} \leq E
	\] 
	for some constant $E>0$ depending on mass and energy. Second, by performing localized virial estimates as in Step 3 of the proof of Theorem \ref{theo-gwp} and using the trivial bound
	\[
	G_\mu(\chi_Ru(t)) \geq C\int |x|^{-b}|\chi_R(x) u(t,x)|^{\alpha+2} dx, \quad \forall t\in \R,
	\]
	we prove the same space-time estimates as in \eqref{space-time-est-non-rad} and \eqref{space-time-est-rad}. Finally, the energy scattering follows by the same argument as in the proof of Theorem \ref{theo-scat}. 		 
\end{proof}

\section*{Acknowledgments}
V. D. D. was supported in part by the European Union's Horizon 2020 Research and Innovation Programme (Grant agreement CORFRONMAT No. 758620, PI: Nicolas Rougerie). S. K. was supported by Labex CEMPI grant (ANR-11-LABX-0007-01) and ANR grant ODA (ANR-18-CE40- 0020-01). 

\appendix

\section{Equivalent thresholds} \label{appendix}
Let $\mu=0$, $\omega>0$, and $Q_\omega$ be a minimizer of $m_{0,\omega}$. We introduce 
\[
\Ac^+_{0,\omega}:= \left\{f\in H^2(\R^d): S_{0,\omega}(f)<m_{0,\omega}, G_0(f)\geq 0\right\}
\]
and
\[
\Bc^+:= \left\{f \in H^2(\R^d) \left|  \begin{array}{lcl}
E_0(f) (M(f))^{\sigc} &<& E_0(Q_1) (M(Q_1))^{\sigc} \\
\|\Delta f\|_{L^2}\|f\|^{\sigc}_{L^2} &<& \|\Delta Q_1\|_{L^2} \|Q_1\|^{\sigc}_{L^2}
\end{array}
\right.
\right\}
\]
with $\sigc:=\frac{2-\gamc}{\gamc} = \frac{8-2b-(d-4)\alpha}{d\alpha-8+2b}$. The main purpose of this appendix is to prove the following equivalent threshold.

\begin{proposition}[\bf Threshold equivalence] \label{prop-equi-thre} ~\\
	Let $d\geq 1$, $\mu=0$, $0<b<\min\{d,4\}$, $\alpha>\frac{8-2b}{d}$, $\alpha<\frac{8-2b}{d-4}$, and $\omega>0$. Then we have 
	\[
	\bigcup_{\omega>0} \Ac^+_{0,\omega} = \Bc^+.
	\]
\end{proposition}

Before proving this result, we need some preparations.

\begin{lemma}[\bf Gagliardo--Nirenberg inequality] \label{lem-GN-ineq} ~\\
	Let $d\geq 1$, $0<b<\min\{d,4\}$, $\alpha>0$, and $\alpha<\frac{8-2b}{d-4}$ if $d\geq 5$. Then the following sharp Gagliardo--Nirenberg inequality holds true
	\begin{align} \label{GN-ineq}
	\int_{\R^d} |x|^{-b} |f(x)|^{\alpha+2}dx \leq C_{\opt} \|\Delta f\|^{\frac{d\alpha+2b}{4}}_{L^2} \|f\|^{\frac{8-2b-(d-4)\alpha}{4}}_{L^2}, \quad \forall f\in H^2(\R^d).
	\end{align}
	The optimal constant is obtained by a non-trivial solution to 
	\begin{align}\label{elli-equa-0}
	\Delta^2 f+ f -|x|^{-b}|f|^\alpha f=0.
	\end{align}
\end{lemma}

\begin{proof}[Proof of Lemma \ref{lem-GN-ineq}]
	Due to the appearance of an biharmonic operator, the method using the Schwarz symmetrization to reduce the problem to the radial setting as in \cite{Farah} does not work. We proceed differently. Denote
	\[
	W(f):= \int |x|^{-b} |f(x)|^{\alpha+2} dx \div \left(\|\Delta f\|^{\frac{d\alpha+2b}{4}}_{L^2} \|f\|^{\frac{8-2b-(d-4)\alpha}{4}}_{L^2} \right)
	\]
	the Weinstein functional. In particular, we have
	\[
	C_{\opt} = \sup \left\{W(f) : f\in H^2(\R^d) \backslash \{0\}\right\}>0.
	\]
	Let $(f_n)_n$ be an optimizing sequence for $C_{\opt}$. As $W(f)$ is invariant under the scaling $f_{\lambda, \zeta}(x)=\lambda f(\zeta x)$ for all $\lambda, \zeta>0$, we can assume that
	\begin{align} \label{prop-fn}
	\|\Delta f_n\|_{L^2}=\|f_n\|_{L^2}=1, \quad \int|x|^{-b} |f_n(x)|^{\alpha+2} dx = W(f_n) \to C_{\opt}.
	\end{align}
	Since $(f_n)_n$ is bounded in $H^2(\R^d)$, the same argument as in the proof of Proposition \ref{prop-mini-prob} shows the existence of $g\in H^2(\R^d)$ such that up to a subsequence, $f_n \rightharpoonup g$ weakly in $H^2(\R^d)$ and
	\[
	\int |x|^{-b} |f_n(x)|^{\alpha+2} dx \to \int |x|^{-b} |g(x)|^{\alpha+2} dx \text{ as } n \to \infty.
	\] 
	Thanks to \eqref{prop-fn} and $C_{\opt}>0$, we have $g\ne 0$. By the weak convergence in $H^2(\R^d)$, we have
	\[
	C_{\opt} \geq W(g) \geq \liminf_{n\to \infty} W(f_n) = C_{\opt}.
	\]
	This shows that $g$ is an optimizer for $C_{\opt}$. In particular, we have
	\[
	\left.\frac{d}{d\vareps}\right|_{\vareps=0} W(g+\vareps \chi) =0, \quad \forall \chi \in C^\infty_0(\R^d).
	\] 
	A direct computation shows
	\[
	\rea \int \( A \Delta^2 g + B g- C |x|^{-b} |g|^\alpha g \) \overline{\chi} dx =0,
	\]
	where
	\[
	A= \frac{(d\alpha+2b) C_{\opt}}{2\|\Delta g\|_{L^2}}, \quad B=\frac{(8-2b-(d-4)\alpha) C_{\opt}}{2\|g\|_{L^2}}, \quad C= \frac{\alpha+2}{\|\Delta g\|^{\frac{d\alpha+2b}{4}}_{L^2}\|g\|^{\frac{8-2b-(d-4)\alpha}{4}}_{L^2}}.
	\]
	Replacing $\chi$ by $i\chi$, we have the same identity with the imaginary part instead of the real part. In particular, we have
	\[
	A \Delta^2 g + B g- C |x|^{-b} |g|^\alpha g =0
	\]
	in the weak sense. By a change of variable $g(x)=\lambda f(\zeta x)$ with 
	\[
	\zeta^4=\frac{B}{A}, \quad \lambda^\alpha = \frac{B^{\frac{4-b}{4}} A^{\frac{b}{4}}}{C},
	\]
	we see that $C_{\opt}= W(f)$ and $f$ is a non-trivial solution to \eqref{elli-equa-0}. The proof is complete.
\end{proof}

\begin{lemma} \label{lem-Q1}
	Let $Q_1$ be a minimizer for $m_{0,1}$. Then $Q_1$ is an optimizer for \eqref{GN-ineq}. 
\end{lemma}
\begin{proof}[Proof of Lemma \ref{lem-Q1}]
	By Lemma \ref{lem-GN-ineq}, there exists an optimizer $f$ for \eqref{GN-ineq} which is a non-trivial solution to \eqref{elli-equa-0}. We have the following Pohozaev identities
	\begin{align} \label{poho-iden}
	\|\Delta f\|^2_{L^2} =\frac{d\alpha+2b}{4(\alpha+2)} \int |x|^{-b}|f(x)|^{\alpha+2} dx = \frac{d\alpha+2b}{8-2b-(d-4)\alpha} \|f\|^2_{L^2}.
	\end{align}
	In particular, we have
	\[
	S_{0,1}(f) = \frac{\alpha}{2(\alpha+2)} \int |x|^{-b}|f(x)|^{\alpha+2} dx
	\]
	and
	\[
	W(f)= \(\frac{4(\alpha+2)}{d\alpha+2b}\)^{\frac{d\alpha+2b}{8}} \(\frac{4(\alpha+2)}{8-2b-(d-4)\alpha}\)^{\frac{8-2b-(d-4)\alpha}{8}} \(\int|x|^{-b}|f(x)|^{\alpha+2} dx\)^{-\frac{\alpha}{2}}.
	\]
	We have
	\begin{align*}
	W(Q_1) &\leq C_{\opt} =W(f)\\
	&= \(\frac{4(\alpha+2)}{d\alpha+2b}\)^{\frac{d\alpha+2b}{8}} \(\frac{4(\alpha+2)}{8-2b-(d-4)\alpha}\)^{\frac{8-2b-(d-4)\alpha}{8}} \(\frac{\alpha}{2(\alpha+2)}\)^{\frac{\alpha}{2}} \(S_{0,1}(f)\)^{-\frac{\alpha}{2}} \\
	&\leq \(\frac{4(\alpha+2)}{d\alpha+2b}\)^{\frac{d\alpha+2b}{8}} \(\frac{4(\alpha+2)}{8-2b-(d-4)\alpha}\)^{\frac{8-2b-(d-4)\alpha}{8}} \(\frac{\alpha}{2(\alpha+2)}\)^{\frac{\alpha}{2}} \(S_{0,1}(Q_1)\)^{-\frac{\alpha}{2}} \\
	&=W(Q_1).
	\end{align*}
	In the third line, we have used \eqref{chara-m-omega} to get
	\[
	S_{0,1}(f) \geq m_{0,1} =S_{0,1}(Q_1).
	\]
	In the last line, we use the fact that $Q_1$ is a solution to \eqref{elli-equa-0}. This shows that $Q_1$ is an optimizer for \eqref{GN-ineq}. 
\end{proof}

\begin{remark}
	In the definition of $\Bc^+$, the terms $E_0(Q_1) (M(Q_1))^{\sigc}$ and $\|\Delta Q_1\|_{L^2} \|Q_1\|_{L^2}^{\sigc}$ are independent from the choice of $Q_1$. In fact, since $Q_1$ is a non-trivial solution to \eqref{elli-equa-0}, we can use the Pohozaev identities \eqref{poho-iden} to express $E_0(Q_1), M(Q_1)$, and $\|\Delta Q_1\|_{L^2}$ in terms of $S_{0,1}(Q_1)=m_{0,1}$. 
\end{remark}

We are now able to prove the equivalent threshold given in Proposition \ref{prop-equi-thre}.

\begin{proof}[Proof of Proposition \ref{prop-equi-thre}]
	We follow an idea of Hamano \cite[Appendix]{Hamano}. Let $Q_\omega$ be a minimizer for $m_{0,\omega}$ and set $Q_\omega(x) = \omega^{\frac{4-b}{4\alpha}} Q_1\(\omega^{\frac{1}{4}} x\)$. We observe that $Q_1$ is a minimizer for $m_{0,1}$. In fact, it follows directly from the fact that $S_{0,\omega}(Q_\omega) = \omega^{\frac{8-2b-(d-4)\alpha}{4\alpha}} S_{0,1}(Q_1)$, and $Q_\omega$ solves
	\[
	\Delta^2 f + \omega f - |x|^{-b}|f|^\alpha f=0
	\]
	if and only if $Q_1$ solves \eqref{elli-equa-0}.

	Now we consider
	\[
	F(\omega):= S_{0,\omega}(Q_\omega) - S_{0,\omega}(f) = \omega^{\frac{8-2b-(d-4)\alpha}{4\alpha}} S_{0,1}(Q_1) - \frac{\omega}{2}M(f)-E_0(f).
	\]
	We have
	\[
	F'(\omega) = \frac{8-2b-(d-4)\alpha}{4\alpha} \omega^{-\frac{d\alpha-8+2b}{4\alpha}} S_{0,1}(Q_1) - \frac{1}{2}M(f).
	\]
	As $\alpha>\frac{8-2b}{d}$ and $\alpha<\frac{8-2b}{d-4}$ if $d\geq 5$, the equation $F'(\omega)=0$ admits a unique positive solution
	\[
	\omega_0 = \(\frac{(8-2b-(d-4)\alpha)}{2\alpha} \frac{S_{0,1}(Q_1)}{M(f)}\)^{\frac{4\alpha}{d\alpha-8+2b}}
	\]
	and $F(\omega_0) = \max_{\omega>0} F(\omega)$. We see that
	\begin{align*}
	&\exists \omega>0 : F(\omega)>0 \\
	&\quad\Longleftrightarrow F(\omega_0)>0 \\
	&\quad\Longleftrightarrow \omega_0 \( \omega_0^{-\frac{d\alpha-8+2b}{4\alpha}} S_{0,1}(Q_1) -\frac{1}{2}M(f)\) -E_0(f) >0 \\
	&\quad\Longleftrightarrow \(\frac{(8-2b-(d-4)\alpha)}{2\alpha} \frac{S_{0,1}(Q_1)}{M(f)}\)^{\frac{4\alpha}{d\alpha-8+2b}} \frac{d\alpha-8+2b}{2(8-2b-(d-4)\alpha)} M(f)-E_0(f) >0 \\
	&\quad\Longleftrightarrow \frac{d\alpha-8+2b}{4\alpha} S_{0,1}(Q_1) \(\frac{8-2b-(d-4)\alpha}{2\alpha} S_{0,1}(Q_1)\)^{\sigc} > E_0(f)(M(f))^{\sigc} \\
	&\quad\Longleftrightarrow E_0(Q_1)(M(Q_1))^{\sigc} > E_0(f)(M(f))^{\sigc}.
	\end{align*}
	Here we have used \eqref{poho-iden} to get
	\begin{align*}
	E_0(Q_1) &=\frac{d\alpha-8+2b}{2(8-2b-(d-4)\alpha)} M(Q_1), \\ 
	S_{0,1}(Q_1) &= \frac{2\alpha}{8-2b-(d-4)\alpha}M(Q_1)=\frac{4\alpha}{d\alpha-8+2b}E_0(Q_1).
	\end{align*}
	Thus we have proved that
	\begin{align} \label{equi-thre-prof}
	\bigcup_{\omega>0} \{f \in H^2(\R^d): S_{0,\omega}(f)<m_{0,\omega}\} = \left\{f\in H^2(\R^d) : E_0(f)(M(f))^{\sigc}<E_0(Q_1)(M(Q_1))^{\sigc}\right\}.
	\end{align}
	Next, by \eqref{GN-ineq}, we have
	\begin{align*}
	E_0(f)(M(f))^{\sigc} &\geq \frac{1}{2} \(\|\Delta f\|_{L^2}\|f\|^{\sigc}_{L^2}\)^2 - \frac{C_{\opt}}{\alpha+2} \(\|\Delta f\|_{L^2}\|f\|^{\sigc}_{L^2}\)^{\frac{d\alpha+2b}{4}} \\
	&=: G\(\|\Delta f\|_{L^2}\|f\|^{\sigc}_{L^2}\),
	\end{align*}
	where $G(\lambda)= \frac{1}{2}\lambda^2 - \frac{C_{\opt}}{\alpha+2} \lambda^{\frac{d\alpha+2b}{4}}$. As $Q_1$ is an optimizer for the Gagliardo--Nirenberg inequality (see Lemma \ref{lem-Q1}), we infer from \eqref{poho-iden} that
	\begin{align*}
	C_{\opt} &= \int |x|^{-b}|Q_1(x)|^{\alpha+2} dx \div \( \|\Delta Q_1\|^{\frac{d\alpha+2b}{4}}_{L^2} \|Q_1\|^{\frac{8-2b-(d-4)\alpha}{4}}_{L^2}\) \\
	&= \frac{4(\alpha+2)}{d\alpha+2b} \(\|\Delta Q_1\|_{L^2} \|Q_1\|^{\sigc}_{L^2}\)^{-\frac{d\alpha-8+2b}{4}}.
	\end{align*}
	In particular, $G$ attains its maximum at 
	\[
	\lambda_0 = \|\Delta Q_1\|_{L^2}\|Q_1\|^{\sigc}_{L^2}.
	\]
	Similarly, we have
	\begin{align*}
	G_0(f) (M(f))^{\sigc} &\geq \(\|\Delta f\|_{L^2} \|f\|^{\sigc}_{L^2}\)^2 -\frac{(d\alpha+2b)C_{\opt}}{2(\alpha+2)} \(\|\Delta f\|_{L^2} \|f\|_{L^2}^{\sigc}\)^{\frac{d\alpha+2b}{4}}\\
	&=:H\(\|\Delta f\|_{L^2} \|f\|_{L^2}^{\sigc}\), 
	\end{align*}
	where $H(\lambda)=\lambda^2 - \frac{(d\alpha+2b)C_{\opt}}{2(\alpha+2)} \lambda^{\frac{d\alpha+2b}{4}}$. Observe that $H(\lambda)=\lambda G'(\lambda)$, hence
	\[
	H(\lambda) \left\{
	\begin{array}{cl}
	> 0 &\text{if } 0< \lambda<\lambda_0, \\
	=0 &\text{if } \lambda=\lambda_0, \\
	<0 &\text{if } \lambda>\lambda_0.
	\end{array}
	\right.
	\]
	
	If $f\in \Bc^+$, then we have $\lambda:=\|\Delta f\|_{L^2} \|f\|_{L^2}^{\sigc}<\|\Delta Q_1\|_{L^2} \|Q_1\|_{L^2}^{\sigc}=\lambda_0$. We infer that $H(\lambda)\geq 0$ or $G_\mu(f) \geq 0$. This combined with \eqref{equi-thre-prof} show that $f\in \bigcup_{\omega>0} \Ac^+_{0,\omega}$.
	
	If $f\in \Ac^+_{0,\omega}$ for some $\omega>0$, then we have $G_0(f) \geq 0$. It follows that 
	\[
	0\leq G_0(f) (M(f))^{\sigc}= \frac{d\alpha+2b}{4} E_0(f) (M(f))^{\sigc} - \frac{d\alpha-8+2b}{4} \(\|\Delta f\|_{L^2}\|f\|_{L^2}^{\sigc}\)^2
	\]
	which implies
	\begin{align*}
	\frac{d\alpha-8+2b}{4} \(\|\Delta f\|_{L^2}\|f\|_{L^2}^{\sigc}\)^2 &= \frac{d\alpha+2b}{4} E_0(f) (M(f))^{\sigc} - G_0(f) (M(f))^{\sigc} \\
	&\leq \frac{d\alpha+2b}{4} E_0(f)(M(f))^{\sigc} \\
	&< \frac{d\alpha+2b}{4} E_0(Q_1)(M(Q_1))^{\sigc} \\
	&= \frac{d\alpha-8+2b}{4} \(\|\Delta Q_1\|_{L^2} \|Q_1\|_{L^2}^{\sigc}\)^2.  
	\end{align*}
	This shows that
	\[
	\|\Delta f\|_{L^2}\|f\|^{\sigc}_{L^2} < \|\Delta Q_1\|_{L^2} \|Q_1\|_{L^2}^{\sigc}
	\]
	which together with \eqref{equi-thre-prof} imply that $f\in \Bc^+$. The proof is complete.
\end{proof}



\end{document}